\documentclass[14pt,a4paper,oneside]{amsart}
\usepackage[english]{babel}
\usepackage[utf8]{inputenc}
\usepackage{amsmath,amssymb}
\usepackage{amsthm,mathtools}
\usepackage[foot]{amsaddr}
\usepackage{graphicx}
\allowdisplaybreaks
\usepackage{bbm}
\usepackage{color}

\newcommand{\blue}[1]{{\color{blue}#1}}

\usepackage{geometry}
\usepackage{tikz}
\usepackage{tikz-cd}
\usetikzlibrary{intersections,decorations.markings,calc,angles}
\usepackage{mathrsfs}
\usepackage[colorinlistoftodos, textwidth=25mm]{todonotes}
\usepackage{enumitem}

\usepackage{hyperref}
\hypersetup{hypertex=true,colorlinks=true,linkcolor=cyan,anchorcolor=cyan,citecolor=cyan}
\let\C\undefined

\newcommand{\C}{\mathbb{C}}

\newcommand{\inv}{^{-1}}
\newcommand{\mbc}{\mathbb{C}}

\newcommand{\mbe}{\mathbb{E}}

\newcommand{\mbr}{\mathbb{R}}

\newcommand{\mca}{\mathcal{A}}
\newcommand{\mcb}{\mathcal{B}}

\newcommand{\mce}{\mathcal{E}}

\newcommand{\mcr}{\mathcal{R}}

\newcommand{\mfpa}{\mathfrak{a}}

\newcommand{\mmd}{\mathrm{d}}
\newcommand{\mme}{\mathrm{e}}
\newcommand{\mmi}{\mathrm{i}}

\newcommand{\ol}{\overline}

\newcommand{\fot}{\frac{1}{2}}
\newcommand{\fq}{\frac{1}{q}}

\DeclareMathOperator{\IM}{Im}
\DeclareMathOperator{\RE}{Re}

\DeclareMathOperator{\csch}{csch}

\newmuskip\pFqmuskip
\newcommand*\pFq[6][8]{%

}

\usepackage{verbatim}

\theoremstyle{plain}
\newtheorem{thm}{Theorem}[section]
\newtheorem{lem}[thm]{Lemma}
\newtheorem{prop}[thm]{Proposition}

\newtheorem{cor}[thm]{Corollary}

\theoremstyle{definition}

\theoremstyle{remark}

\title{Distribution of Dirichlet $L$-functions}
\author{Zikang Dong}
\author{Weijia Wang}
\author{Hao Zhang}
\address[Zikang Dong]{School of Mathematical Sciences, Tongji University, Shanghai 200092, P. R. China}
\address[Weijia Wang]{Yanqi Lake Beijing Institute of Mathematical Sciences and Applications $\&$ Yau Mathematical Sciences Center, Tsinghua University, Beijing 101408, P. R. China}
\address[Hao Zhang]{School of Mathematics, Hunan University, Changsha 410082, P. R. China}
\email{zikangdong@gmail.com}
\email{weijiawang@tsinghua.edu.cn}
\email{zhanghaomath@hnu.edu.cn}

\begin{document}
	
	\maketitle
	\section*{Abstract}
	In this article, we study the distribution of values of Dirichlet $L$-functions, the distribution of values of the random models for Dirichlet $L$-functions, and the discrepancy between these two kinds of distributions. For each question, we consider the cases of $\frac12<\RE s<1$ and $\RE s=1$ separately.
	\bigskip
	\section{\blue{Introduction}}
	The analytic theory of $L$-functions is a central part of modern number theory. The study of distribution of values of $L$-functions is an important topic in the analytic theory of $L$-functions. In \cite{BohrJessen}, Bohr and Jessen introduced a probability treatment to study the distribution of values of the Riemann zeta function. They proved that $\log \zeta(\sigma+\mmi t)$ has a continuous limiting distribution for any $\sigma>\frac12$.  On the critical line, we have Selberg's central limit theorem. On the 1-line, Granville and Soundararajan \cite{GS06} studied the distribution of $|\zeta(1+\mmi t)|$, which is asymptotically a double exponentially decreasing function. Their method was also adjusted to apply to the distribution of values on the 1-line of other $L$-functions. In 2003, they \cite{GS03} studied the distribution of the Dirichlet $L$-functions 
	of quadratic characters $L(1,\chi_d)$, which proves part of Montgomery and Vaughan's conjecture in \cite{MV199}. In 2007, Wu \cite{Wu2007} improved this result by giving a high order expansion in the exponent of the distribution function. 
	In 2008, Liu, Royer and Wu \cite{LRW2008} studied the distribution of a kind of symmetric power $L$-functions. 
	In 2010, Lamzouri \cite{La2010} studied a generalized $L$-function 
	which can cover the results of \cite{GS03, LRW2008}. In the critical strip $\frac12<\RE( s)<1$, Lamzouri \cite{La2011} in 2011 studied the distribution of $\log|\zeta(\sigma+\mmi t)|$ with any fixed $\frac12<\sigma<1$ and also got the asymptotic distribution function. 
	In 2019, Lamzouri, Lester and Radziwi\l\l \   \cite{LLR19} studied the discrepancy between the  distribution of $\log\zeta(\sigma+\mmi t)$ and that of their random models. Later in 2021, Xiao and Zhai \cite{XZ21} generalized this result to automorphic $L$-functions.
	
	For each prime $p$, $X(p)$ denotes the independent random variable uniformly distributed on the unit circle. Then the product
	\[L(\sigma,X):=\prod_p\bigg(1-\frac{X(p)}{p^\sigma}\bigg)^{-1}\]
	converges almost surely for any $\sigma>\frac12$. This random $L$-function turns out to be a very good model for the Riemann zeta function $\zeta(\sigma+{\rm i}t)$, where $t\in[T,2T]$ as $T\to\infty$.
	
	Now we turn our attention to Dirichlet $L$-functions. Let $q$ be a large prime number, and $\chi({\rm mod\;}q)$ be any character modulo $q$. As $\chi$ varies modulo $q$, the values of $L(\sigma,\chi)$ behavior similarly to $\zeta(\sigma+{\rm i}t)$. So the random $L$-function $L(\sigma,X)$ is a good model for $L(\sigma,\chi)$ as well. For every real positive number $\tau$ and any fixed $\frac12<\sigma<1$, we define the distribution functions  separately by:
	$$\Phi_q(\tau)=\Phi_q(\sigma,\tau):=\frac{1}{q}\#\{\chi({\rm mod}\; q):\log|L(\sigma,\chi)|>\tau\},$$
	and 
	$$\Psi(\tau)=\Psi(\sigma,\tau):={\rm Prob}(\log|L(\sigma,X)|>\tau).$$
	In 2011, Lamzouri \cite{La2011} showed that there is a constant ${\mathfrak a}_0$ such that 
	\[\Phi(\tau)=\exp\bigg(-(\tau\log^{\sigma}\tau)^{\frac{1}{1-\sigma}}
	\bigg\{{\mathfrak a}_0
	+ O\bigg(\frac1{\sqrt{\log\tau}} 
	+ \bigg(\frac{(\tau\log\tau)^{\frac{1}{1-\sigma}}}{\log q}\bigg)^{\sigma-\frac{1}{2}}\bigg)\bigg\}\bigg),\]
	holds for $1\ll\tau<b(\sigma)(\log q)^{1-\sigma}(\log_2q)^{-1} $ with some constant $b(\sigma)$.
	
	For $\sigma=1$, we define the distribution functions slightly differently:
	$$\Phi_{1,q}(\tau):=\frac{1}{q}\#\{\chi({\rm mod}\; q):|L(1,\chi)|>{\rm e}^\gamma\tau\},$$
	and 
	$$\Psi_1(\tau):={\rm Prob}(|L(1,X)|>{\rm e}^\gamma\tau).$$
	In 2006, Granville and Soundararajan \cite{GS06} showed that uniformly for $1\ll\tau<\log_2q-20$,
	\[\Phi_{1,q}(\tau)=\exp\bigg(\!-\frac{{\rm e}^{\tau-A_0-1}}{\tau}
	\bigg\{1
	+ O\bigg(\frac{1}{\sqrt\tau}+\sqrt{\frac{{\rm e}^{\tau}}{\log q}}\bigg)\bigg\}\bigg).\]
	Firstly, we give  asymptotic formulae  of $\Psi(\tau)$ and $\Psi_1(\tau)$ with a high order expansion in the exponent, which improve \cite[Corollary 7.7]{LLR19}.
	\begin{thm}\label{thm1.1}
		Let $1/2<\sigma<1$ be fixed. Then for any integer $N\ge1$ there exist computable polynomials $\mfpa_0(\cdot),\dots,\mfpa_N(\cdot)$ with $\deg \mfpa_i\leq i$ which depend only on $\sigma$ and $N$, such that for any $\tau\geq 2$, we have
		$$\Psi(\tau)=\exp\left(-(\tau\log^{\sigma}\tau)^{\frac{1}{1-\sigma}}\left(\sum_{n=0}^{N}\frac{\mathfrak{a}_n(\log_2\tau)}{(\log\tau)^n}+ O\left(\frac{\log_2\tau}{\log\tau}\right)^{\!N+1}\right)\right),$$
		with $\mfpa_0>0$.
		When $\sigma=1$,	there is a sequence of real numbers $\{\mathfrak{b}_n\}_{n\ge 1}$ such that for any integer $N\ge 1$ we have
		\begin{equation*}
			\Psi_1(\tau)=\exp\bigg(\!-\frac{{\rm e}^{\tau-A_0-1}}{\tau}
			\bigg\{1+\sum_{n=1}^N \frac{\mathfrak{b}_n}{\tau^n}
			+ O_N\bigg(\frac{1}{\tau^{N+1}}\bigg)\bigg\}\bigg)
		\end{equation*}
		uniformly for  $  \tau\ge2$,
		where $A_0=A_0^{(0)}$ is defined in \eqref{defofA_n}.
	\end{thm}
	Now we present the relations between $\Psi(\tau)$ and $\Phi_q(\tau)$, and between $\Psi_1(\tau)$ and $\Phi_{1,q}(\tau)$, which are similar to \cite[Theorem 1.3]{LLR19}
	\begin{thm}\label{thm1.2}
		Let $\frac12<\sigma<1$ be fixed. There exists a positive constant $b(\sigma)$ such that for $3\le\tau \leq b(\sigma)(\log q)^{1-\sigma}(\log_2 q)^{-1}$ we have
		\[\Phi_q(\tau)=\Psi(\tau)\bigg(1+O\bigg(\frac{(\tau\log \tau)^{\frac{\sigma}{1-\sigma}}\log_2 q}{(\log q)^{\sigma}}\bigg)\bigg).\]
		When $\sigma=1$, there exists a constant $b$ such that uniformly for $1\ll\tau\le\log_2q-b$ we have 
		\[\Phi_{1,q}(\tau)=\Psi_1(\tau)\left(1+O\left(\frac{\mme^{\tau}\log_2 q}{\tau\log q}\right)\right).\]
	\end{thm}
	
	At last, we study the discrepancy  between the distribution of Dirichlet $L$-functions and that of their random models. Let ${\mathcal R}$ be any rectangle with sides parallel to the coordinate axis. Let
	$$\Phi_q({\mathcal R}):=\frac{1}{q}\#\{\chi({\rm mod}\; q):\log L(\sigma,\chi)\in{\mathcal R}\},$$
	and 
	$$\Psi({\mathcal R}):={\rm Prob}\left(\log L(\sigma,X)\in {\mathcal R}\right).$$
	The discrepancy between the above two probabilities is defined by
	$$D_\sigma( q):=\sup_{\mathcal R}|\Phi_q({\mathcal R})-\Psi({\mathcal R})|,$$
	where $\mcr$ runs through all the rectangles with sides parallel to the coordinate axis. We have the following theorem.
	\begin{thm}\label{thm1.3}
		Let $\frac12<\sigma<1$ be fixed and $q$ be a large prime number. Then we have
		\[D_\sigma(q)\ll\frac{1}{(\log q)^\sigma}.\]
		When $\sigma=1$, we have
		\[D_1(q)\ll\frac{(\log_2q)^2}{\log q}.\]
	\end{thm}
	
	%

	This article is organized as follows.  In \S\ref{chap2}, we introduce some preliminary lemmas. In \S\ref{chap3}, we study the distribution of random models and prove Theorem \ref{thm1.1}. In \S\ref{chap4}, we study large deviations between the distribution of Dirichlet $L$-functions and that of their random models. We prove Theorem \ref{thm1.2}. In \S\ref{chap5}, we study the discrepancy bound for Dirichlet $L$-functions and prove Theorem \ref{thm1.3}.
	\section{\blue{Preliminary Lemmas}}\label{chap2}
	Firstly we state a lemma which  approximates Dirichlet $L$-functions by their truncating sums.
	\begin{lem}\label{l6}   Let $s=\sigma+{\rm i}t$ with $|t|\le 3q$, and let $y\ge 2$ be a real number. Let $1/2\leq\sigma_0<\sigma$, and suppose that the rectangle $\{z:\sigma_0<\RE z\le 1,|\IM z-t|\le y+3\}$ contains no zeros of $L(z,\chi)$. Then
		$$\log |L(s,\chi)|\ll \frac{\log q}{\sigma-\sigma_0}.$$
		Further, now putting $\sigma_1={\rm min}\left(\sigma_0+\frac{1}{\log y},\frac{\sigma+\sigma_0}{2}\right)$, we have
		$$\log L(\sigma+{\rm i}t,\chi)=\sum_{n=2}^y \frac{\Lambda(n)\chi(n)}{n^{\sigma+{\rm i}t}\log n}+O\bigg(\frac{\log  q}{(\sigma_1-\sigma_0)^2}y^{\sigma_1-\sigma}\bigg).$$ 
	\end{lem}
	\begin{proof}
		See \cite[Lemma 2.1]{GS03}.
	\end{proof}

	We need the following zero-density estimates for Dirichlet $L$-functions.
	
	\begin{lem}\label{l7} Let $1/2\leq\sigma\leq 1$ and $N(\sigma,T,\chi)$ denote the number of zeros of $L(s,\chi)$ in the region $\RE s\ge\sigma$ and $|\IM s|\le T$. Then we have
		$$\sum\limits_{\chi({\rm mod}\, q)}N(\sigma,T,\chi)\ll(qT)^{{\frac{3-3\sigma}{2-\sigma}}}(\log qT)^{14}.$$
	\end{lem}
	\begin{proof}
		See \cite[Theorem 12.1]{Mon71}.
	\end{proof}
	For any $y>1$ and $1/2<\sigma\leq 1$, let 
	\[R_y(\sigma,\chi):=\sum_{p^n\leq y}\frac{\chi(p)^n}{np^{n\sigma}}.\]
	With the help of Lemma \ref{l6} and Lemma \ref{l7}, we can show that with very few exceptions, the logarithms of the $L$-functions can be approximated by $R_y(\sigma,\chi)$.

	\begin{lem}\label{l8} {\it   Let $q$ be a large prime number and $1/2<\sigma\leq 1$. Let $(\log q)^{A(\sigma)}\le  y\le q^{a(\sigma)}$ be a real number, where $0<a(\sigma)<\frac{4\sigma-2}{7-2\sigma}<\frac{4}{2\sigma-1}<A(\sigma)$ are any constants. Then we have 
			$$\log L(\sigma,\chi)=R_y(\sigma,\chi)+O\Big(y^{{\frac{1-2\sigma}{4}}}(\log y)^2\log q\Big)$$
			for all but at most $q^{\frac{9-6\sigma}{7-2\sigma}}y(\log q)^{14}$ primitive characters $\chi\,({\rm mod}\,q)$.}
	\end{lem}
	\begin{proof}
		This follows from Lemma \ref{l6} and Lemma \ref{l7} with the choice of $\sigma_0=(\sigma+1/2)/2$.
	\end{proof}
	
	Now we need to give an estimation on the power of $R_y(\sigma,\chi)$. The idea is to divide $R_y(\sigma,\chi)$ into three parts. The following Lemma gives the bound of the main part.
	
	\begin{lem}\label{lem:psigma2k}
		Let $q$ be a prime number and $1/2<\sigma\leq 1$. Let $2\leq y<z$ be real numbers. Then for all positive integers $k$ with $1\leq k\leq \log q/(2\log z)$, we have 
		\[\frac{1}{q}\sum_{\chi({\rm mod}\  q)}\left|\sum_{y\leq p\leq z}\frac{\chi(p)}{p^\sigma}\right|^{2k}\ll k!\left(\sum_{y\leq p\leq z}\frac{1}{p^{2\sigma}}\right)^k+O(q^{-1/2}).\]
	\end{lem}
	\begin{proof}
		See~\cite[Lemma $4.4$]{La2011}
	\end{proof}
	
	\begin{lem}\label{lem:1qchi2k}
		Let $q$ be a prime number and $1/2<\sigma<1$. Suppose that $y=(\log q)^a$ for some $a\geq 1$ and $k$ is an integer with $1<k<\frac{\log q}{3a\log_2 q}$. Then there exists a constant $c_1(\sigma)>0$ such that 
		\[\frac{1}{q}\sum_{\chi({\rm mod}\  q)}|R_y(\sigma,\chi)|^{2k}\ll\left(\frac{c_1(\sigma)k^{1-\sigma}}{(\log k)^{\sigma}}\right)^{2k}.\]
	\end{lem}
	\begin{proof}
		The power mean inequality gives
		\[\frac{1}{q}\sum_{\chi({\rm mod}\  q)}|R_y(\sigma,\chi)|^{2k}\leq \frac{3^{2k}}{q}\sum_{\chi}\left(\left(\sum_{p\leq k\log k}\frac{1}{p^{\sigma}}\right)^{2k}+\left|\sum_{k\log k\leq p\leq y}\frac{\chi(p)}{p^{\sigma}}\right|^{2k}+\left|\sum_{n\geq 2,p^n\leq y}\frac{1}{np^{n\sigma}}\right|^{2k}\right).\]
		Note that when $1/2< \sigma< 1$, we have
		\[\sum_{p\leq x}\frac{1}{p^{\sigma}}\ll\frac{x^{1-\sigma}}{(1-\sigma)\log x},\]
		and when $\sigma>1$, we have 
		\begin{equation}\label{eq:psigma1}
			\sum_{x_0\leq p\leq x_1}\frac{1}{p^{\sigma}}\ll\int_{x_0}^\infty \frac{1}{t^{\sigma}}\mmd\pi(t)\ll \frac{x_0^{1-\sigma}}{(\sigma-1)\log x_0}.
		\end{equation}
		So together with Lemma \ref{lem:psigma2k}, we have 
		\[\left(\sum_{p\leq k\log k}\frac{1}{p^{\sigma}}\right)^{2k}\ll \left(\frac{(k\log k)^{1-\sigma}}{(1-\sigma)\log k}\right)^{2k},\]
		and 
		\begin{align*}
			&\frac{1}{q}\sum_{\chi}\left|\sum_{k\log k\leq p\leq y}\frac{\chi(p)}{p^{\sigma}}\right|^{2k}\ll \left(\sum_{k\log k\leq p\leq y}\frac{k}{p^{2\sigma}}\right)^k+O(q^{-1/2})\\
			\ll& \left(\frac{k(k\log k)^{1-2\sigma}}{(2\sigma-1)\log k}\right)^k+O(q^{-1/2}).
		\end{align*}
		Finally, the third term is dominated by $\left(\sum_{m\geq 2}m^{-2\sigma}\right)^{2k}.$ This completes the proof.
	\end{proof}
	The following lemma shows that there is  only a small number of characters $\chi$ such that the values of  $|R_y(\sigma,\chi)|$ are large.

	\begin{lem}\label{lem:AqB}
		Let $q$ be a large prime number and $1/2<\sigma<1$. Suppose that $y=(\log q)^a$ for some $a\geq 1$. We denote  
		\[\mca_q=\{\chi({\rm mod}\  q)\,:\,|R_y(\sigma,\chi)|\geq \frac{(\log q)^{1-\sigma}}{\log_2 q}\}.\]
		Then there exists a constant $c_2(\sigma)>0$ such that 
		\[\frac{\# \mca_q}{q}\ll \exp\left(-\frac{c_2(\sigma)\log q}{\log_2 q}\right).\]
	\end{lem}
	
	\begin{proof}
		It is easy to see that 
		\[\# \mca_q\left(\frac{(\log q)^{1-\sigma}}{\log_2 q}\right)^{2k}\leq \sum_{\chi({\rm mod}\  q)}|R_y(\sigma,\chi)|^{2k}.\]
		We choose $k=[\frac{\log q}{c_1'(\sigma)\log_2 q}]$ where $c_1'(\sigma)=\max\{(1+c_1(\sigma))^{\frac{1}{1-\sigma}},3a\}$, then Lemma \ref{lem:1qchi2k} gives 
		\[\frac{\# \mca_q}{q}\ll \left(\frac{c_1(\sigma)}{c_1'(\sigma)^{1-\sigma}}\right)^{\frac{2\log q}{c_1'(\sigma)\log_2 q}} =\exp\left(-\frac{c_2(\sigma)\log q}{\log_2 q}\right),\]
		where $c_2(\sigma)=\frac{2}{c_1'(\sigma)}\log(\frac{c_1'(\sigma)^{1-\sigma}}{c_1(\sigma)})$.
	\end{proof}
	
	Similar to $R_y(\sigma,\chi)$, let 
	\[R_y(\sigma,X)=\sum_{p^n\leq y}\frac{X(p)^n}{np^{\sigma n}}.\]
	The following lemma gives the relation between $R_y(\sigma,\chi)$ and $R_y(\sigma,X)$. 
	
	\begin{lem}\label{lem:rykl}
		Let $1/2<\sigma<1$ and $q$ be a large prime number.  Suppose that $y=(\log q)^a$ for some $a\geq 1$. Then for any integers $0\leq k,\ell\leq \frac{\log q}{a\log_2 q}$, we have 
		\[\frac{1}{q}\sum_{\chi({\rm mod}\  q)}R_y(\sigma,\chi)^k\overline{R_y(\sigma,\chi)}^\ell=\mbe\left(R_y(\sigma,X)^k\overline{R_y(\sigma,X)}^\ell\right).\]
	\end{lem}
	\begin{proof}
		By expanding the $R_y(\sigma,\chi)$, we have 
		\begin{align*}
			&\frac{1}{q}\sum_{\chi({\rm mod}\  q)}R_y(\sigma,\chi)^k\overline{R_y(\sigma,\chi)}^\ell\\
			=&\fq\sum_{p_1^{n_1},\dots,p_{k+\ell}^{n_{k+\ell}}\leq y}\sum_{\chi({\rm mod}\  q)}\frac{\chi(p_1^{n_1}\dots p_k^{n_k})\overline{\chi(p_{k+1}^{n_{k+1}}\dots p_{k+\ell}^{n_{k+\ell}})}}{n_1p_1^{n_1\sigma}\dots n_{k+\ell}p_{k+\ell}^{n_{k+\ell}\sigma}}.
		\end{align*}
		The orthogonality of Dirichlet character shows that the inner sum vanishes unless $p_1^{n_1}\dots p_k^{n_k}\equiv p_{k+1}^{n_{k+1}}\dots p_{k+\ell}^{n_{k+\ell}}\pmod{q}$. But we note that 
		\[p_1^{n_1}\dots p_k^{n_k}\leq y^k\leq q.\]
		This implies that the inner sum is non-vanishing if and only if $p_1^{n_1}\dots p_k^{n_k}= p_{k+1}^{n_{k+1}}\dots p_{k+\ell}^{n_{k+\ell}}$. This exactly gives $\mbe\left(R_y(\sigma,X)^k\overline{R_y(\sigma,X)}^\ell\right)$.
	\end{proof}
	
	We give a simple estimation on the partial sum which we will often use.
	
	\begin{lem}\label{lem:zct}
		Fix $D>0$ and $1/2\leq \sigma<1$.Let $N=T/(D\log T)$. Then for any $z\leq \frac{1}{2D\mme^2}T^{\sigma}$, we have 
		\[\sum_{n\geq N}\frac{1}{n!}\left(\frac{zT^{1-\sigma}}{\log T}\right)^n\leq \mme^{-N}.\]
	\end{lem}
	\begin{proof}
		The Stirling's approximation shows that $n!\geq (n/\mme)^n$. Then we have 
		\[\sum_{n\geq N}\frac{1}{n!}\left(\frac{zT^{1-\sigma}}{\log T}\right)^n\leq \sum_{n\geq N}\left(\frac{\mme zT^{1-\sigma}}{N{\log T}}\right)^n=\sum_{n\geq N}\left(\mme DzT^{-\sigma}\right)^n.\]
		By the assumption $\mme DzT^{-\sigma}\leq \frac{1}{2\mme}$, we have 
		\[\sum_{n\geq N}\frac{1}{n!}\left(\frac{zT^{1-\sigma}}{\log T}\right)^n\leq\sum_{n\geq N}\frac{1}{(2\mme)^n}\leq \mme^{-N}.\]
	\end{proof}
	
	Next we will investigate the discrepancy between the exponential of $R_y(\sigma,\chi)$ and that of the random variable $R_y(\sigma,X)$. The idea is to expand the exponential function, and the main part is given by Lemma \ref{lem:rykl}. Then we only need to estimate the error terms.
	
	\begin{lem}\label{lem:chiaqc}
		Let $q$ be a large prime number and $1/2<\sigma<1$. Suppose that $y=(\log q)^a$ for some $a\geq 1$. Then there exists a constant $c(\sigma,a)$ such that for any complex numbers $z_1,z_2$ with $|z_1|,|z_2|\leq c(\sigma,a)(\log q)^\sigma$, we have 
		\begin{align*}
			&\fq\sum_{\chi\in \mca_q^c}\exp\left(z_1R_y(\sigma,\chi)+z_2\overline{R_y(\sigma,\chi)}\right)\\
			=&\mbe\left(\exp\left(z_1R_y(\sigma,X)+z_2\overline{R_y(\sigma,X)}\right)\right)+\exp\left(-\frac{c_3(\sigma)\log q}{\log_2 q}\right),
		\end{align*}
		for some constant $c_3(\sigma)>0$, where $\mca_q^c$ is the set of all $\chi ({\rm mod}\ q)$ with $\chi \notin \mca_q$. 
	\end{lem}
	\begin{proof}
		Let $N=\frac{\log q}{2(a+\sigma)\log_2 q}$, then we have 
		\begin{equation}\label{eq:expry}
			\begin{split}
				&\fq\sum_{\chi\in \mca_q^c}\exp\left(z_1R_y(\sigma,\chi)+z_2\overline{R_y(\sigma,\chi)}\right)\\
				=&\fq\sum_{k+\ell\leq N}\sum_{\chi}-\fq\sum_{k+\ell\leq N}\sum_{\chi\in \mca_q}+\fq\sum_{k+\ell\geq N}\sum_{\chi\in \mca_q^c}\frac{z_1^kz_2^\ell R_y(\sigma,\chi)^k\overline{R_y(\sigma,\chi)}^\ell}{k!\ell!}.
			\end{split}
		\end{equation}
		Lemma \ref{lem:rykl} gives the estimation of the first term of  \eqref{eq:expry}:
		\begin{align*}
			&\fq\sum_{k+\ell\leq N}\sum_{\chi}\frac{z_1^kz_2^\ell R_y(\sigma,\chi)^k\overline{R_y(\sigma,\chi)}^\ell}{k!\ell!}=\sum_{k+\ell\leq N}\frac{z_1^kz_2^\ell}{k!\ell!}\mbe(R_y(\sigma,X)^k\overline{R_y(\sigma,X)}^\ell)+O(y^{k+\ell}/q)\\
			=&\mbe\left(\exp\left(z_1R_y(\sigma,X)+z_2\overline{R_y(\sigma,X)}\right)\right)-\sum_{k+\ell\geq N}\frac{z_1^kz_2^\ell}{k!\ell!}\mbe\left(R_y(\sigma,X)^k\overline{R_y(\sigma,X)}^\ell\right)+O\left(\sum_{k+\ell\leq N}\frac{z_1^kz_2^\ell y^{k+\ell}}{k!\ell!q}\right).
		\end{align*}
		By Lemma \ref{lem:1qchi2k}, the second term is bounded by 
		\begin{align*}
			&\sum_{k+\ell\geq N}\frac{(z_1+z_2)^{k+\ell}}{k!\ell!}\left(\frac{c_1(\sigma)(k+\ell)^{1-\sigma}}{(\log (k+\ell))^\sigma}\right)^{k+\ell}=\sum_{n\geq N}\frac{1}{n!}\left(\frac{c_1(\sigma)(z_1+z_2)n^{1-\sigma}}{(\log n)^\sigma}\right)^n\sum_{k+\ell=n}\frac{n!}{k!\ell!}
		\end{align*}
		Since $\sum_{k+\ell=n}\frac{n!}{k!\ell!}=2^n$, so Lemma \ref{lem:zct} shows that the summation above is bounded by $\mme^{-N}$ for some suitable constant $c(\sigma,a)$. For the third term, since $(yz_1)^N\ll \sqrt{q}$, we have 
		\[\sum_{k+\ell\leq N}\frac{z_1^kz_2^\ell y^{k+\ell}}{k!\ell!q}\leq \sum_{n\leq N}\frac{(2(z_1+z_2)y)^n}{qn!}\ll \frac{(y(z_1+z_2))^N}{q}\ll q^{-1/2}.\]
		Next we consider the second term of Eq. \eqref{eq:expry}. By the Cauchy-Schwarz inequality, we have 
		\[\left|\sum_{\chi\in \mca_q}R_y(\sigma,\chi)^k\overline{R_y(\sigma,\chi)}^\ell\right|\leq \left(\# \mca_q\sum_{\chi}|R_y(\sigma,\chi)|^{k+\ell}\right)^{\frac{1}{2}}.\]
		Together with Lemma \ref{lem:1qchi2k} and Lemma \ref{lem:AqB}, we have 
		\begin{align*}
			&\fq\sum_{k+\ell\leq N}\sum_{\chi\in \mca_q}\frac{z_1^kz_2^\ell R_y(\sigma,\chi)^k\overline{R_y(\sigma,\chi)}^\ell}{k!\ell!}\\
			\ll& \exp\left(-\frac{c_2(\sigma)\log q}{2\log_2 q}\right)\sum_{k+\ell\leq N}\frac{z_1^kz_2^\ell}{k!\ell!}\left(c_1(\sigma)(k+\ell)^{1-\sigma}\right)^{k+\ell}\\
			\ll& \exp\left(-\frac{c_2(\sigma)\log q}{2\log_2 q}\right)\sum_{n\leq N}\frac{1}{n!}(2c_{\sigma}(z_1+z_2)N^{1-\sigma})^n\ll \exp\left(-\frac{c_2(\sigma)\log q}{2\log_2 q}\right)
		\end{align*}
		Finally, we consider the last term of Eq. \eqref{eq:expry}. Since $|R(\sigma,\chi)|<(\log q)^{1-\sigma}/\log_2 q$ when $\chi\in \mca_q^c$, the third term is bounded by
		\begin{align*}
			\sum_{n\geq N}\left(\frac{(z_1+z_2)(\log q)^{1-\sigma}}{\log_2 q}\right)^n\sum_{k+\ell=n}\frac{1}{k!\ell!} =\sum_{n\geq N}\frac{1}{n!}\left(\frac{2(z_1+z_2)(\log q)^{1-\sigma}}{\log_2 q}\right)^n
		\end{align*}
		Then following from Lemma \ref{lem:zct}, it is bounded by $e^{-N}$. This completes the proof.
	\end{proof}
	
	Now we are able to prove the discrepancy between the exponential of $L(\sigma,\chi)$ and that of the random variable $L(\sigma,X)$ via Lemma \ref{lem:chiaqc}. The following Lemma gives the difference between the exponential of $\log L(\sigma,X)$ and the exponential of $R_y(\sigma,X)$.
	
	\begin{lem}\label{lem:yuv}
		Let $y$ be a large positive real number, then for any real numbers $u,v$ with $|u|+|v|\leq y^{\sigma-1/2}$, we have 
		\begin{align*}
			&\mbe\left(\exp\left(\mmi u\RE \log L(\sigma,X)+\mmi v\IM \log L(\sigma,X)\right)\right)\\
			=&\mbe\left(\exp\left(\mmi u\RE R_y(\sigma,X)+\mmi v\IM R_y(\sigma,X)\right)\right)+O((|u|+|v|)/y^{\sigma-1/2}).
		\end{align*}
	\end{lem}
	\begin{proof}
		\cite[Lemma 4.1]{LLR19}
	\end{proof}

	\begin{lem}\label{thm:lchix}
		Let $q$ be a prime number, $A\geq 1$ and $1/2<\sigma<1$. Then there exists a constant $c$ such that for any $|u|,|v|\leq c(\log q)^\sigma$, we have 
		\begin{align*}
			&\fq\sum_{\chi({\rm mod} \ q)}\exp\left(\mmi u\RE\log L(\sigma,\chi)+\mmi v\IM \log L(\sigma,\chi)\right)\\
			=&\mbe\left(\exp\left( \mmi u\RE \log L(\sigma,X)+\mmi v\IM \log L(\sigma,X)\right)\right)+O((\log q)^{-A}).
		\end{align*}
	\end{lem}
	
	\begin{proof}
		Choose $y=(\log q)^{\frac{4A+8}{2\sigma-1}}$, then Lemma \ref{l8} shows that there is a set $\mcb_q$ with $\# \mcb_q\ll q^{\frac{9-6\sigma}{7-2\sigma}}(\log q)^{\frac{4A+8}{2\sigma-1}+14}$ such that 
		\[\log L(\sigma,\chi)=R_y(\sigma,\chi)+O((\log q)^{-A}),\qquad \forall \chi\in \mcb_q^c.\]
		Note that $\frac{9-6\sigma}{7-2\sigma}<1$, so $\fq\# \mcb_q\ll (\log q)^{-A}$.
		So we have
		\begin{align*}
			&\fq\sum_{\chi({\rm mod}\  q)}\exp\left(\mmi u\RE\log L(\sigma,\chi)+\mmi v\IM \log L(\sigma,\chi)\right)\\
			=&\fq\sum_{\chi \in \mcb_q^c}\exp\left(\mmi u\RE R_y(\sigma,\chi)+\mmi v\IM R_y(\sigma,\chi)\right)+O\left((\log q)^{-A}\right)\\
			=&\fq\sum_{\chi ({\rm mod}\  q)}\exp\left(\mmi u\RE R_y(\sigma,\chi)+\mmi v\IM R_y(\sigma,\chi)\right)+O\left((\log q)^{-A}\right)\\
			=&\fq \sum_{\chi\in \mca_q^c}\exp\left(\mmi u\RE R_y(\sigma,\chi)+\mmi v\IM R_y(\sigma,\chi)\right)+O\left((\log q)^{-A}\right).
		\end{align*}
		The last equality follows from Lemma \ref{lem:AqB} and $\exp\left(-\frac{c_2(\sigma)\log q}{\log_2 q}\right)\ll(\log q)^{-A}$. On the other hand, by taking $z_1=(\mmi u+v)/2,z_2=(\mmi u- v)/2$ in Lemma \ref{lem:chiaqc}, the quantity above equals to 
		\[\mbe(\exp(\mmi u\RE R_y(\sigma,X)+\mmi v\IM R_y(\sigma,X)))+O((\log q)^{-A}).\]
		Then this theorem just follows from Lemma \ref{lem:yuv}.
	\end{proof}
	
	Finally, we give the analog of Lemma \ref{thm:lchix} for $\sigma=1$.
	\begin{lem}\label{lem:laml1}
		Let $q$ be a large prime number. Then uniformly for all complex numbers $z_1,z_2$ in the region $|z_1|,|z_2|\leq  \frac{\log q}{50(\log_2 q)^2}$, we have
		\begin{align*}
			&\fq\sum_{\chi({\rm mod} \ q)}\exp\left(\mmi u\RE\log L(1,\chi)+\mmi v\IM \log L(1,\chi)\right)\\
			=&\mbe\left(\exp\left( \mmi u\RE \log L(1,X)+\mmi v\IM \log L(1,X)\right)\right)+O\left(\exp\left(-\frac{\log q}{2\log_2 q}\right)\right).
		\end{align*}
	\end{lem}
	\begin{proof}
		See \cite[Theorem $9.2$]{La2008}
	\end{proof}

	\section{\blue{Distribution of Random models:  Proof of Theorem \ref{thm1.1}}}\label{chap3}
	\subsection{\blue{When $\frac12<\sigma<1$}}
	In this subsection, we will prove the first part of Theorem \ref{thm1.1}. We first give some basic asymptotic properties of Bessel function. The modified Bessel function of first kind is defined to be 
	$$I_0(u):=\sum_{n\ge0}\frac{(u/2)^{2n}}{n!^2}.$$ Let $f(u):=\log I_0(u)$. We have the following properties for $f$ and its derivatives.
	\begin{lem}\label{lem:Bessel2} 
		For $0\leq u\leq 1$, we have
		$$
		f^{(m)}(u) \asymp \begin{cases} u^2 &\text{ if } m=0\\ u &\text{ if } m\geq 1\text{ and $m$ odd}\\ 1 &\text{ if } m\geq 1 \text{ and $m$ even}.\end{cases}
		$$
		And when $u\gg 1$, we have 
		\[f^{(m)}(u)=\begin{dcases} u+O(\log u)& \text{ if } m=0\\ 1+O(u^{-1}) & \text{ if } m=1\\ (-1)^m\frac{(m-1)!}{2}u^{-m}+O(u^{-m-1})&  \text{ if } m\geq 2. \end{dcases}\]
	\end{lem}
	For $z\in \mathbb{C}$, let us put
	\[
	M(z)\coloneqq\log \mbe(|L(\sigma,X)|^z).
	\]
	For any integer $m\geq 0$ and prime number $p$, we define
	$$M_p^{(m)}(z)\coloneqq\frac{\partial^m}{\partial z^{m}}\log \mbe\left(\Big|1-\frac{X(p)}{p^{\sigma}}\Big|^{-z}\right).$$
	
	To prove Theorem \ref{thm1.1}, we need to know the asymptotic property of $M_p^{(m)}$ for every prime number $p$. When $p$ is much larger compared to $\kappa$, the asymptotic behavior of $M_p^{(m)}$ relies on the first few terms in $p^{-\sigma}$. To be precise, we have
	\begin{lem}\label{lem:pzsigma}
		Let $1/2<\sigma\leq 1$ and $\kappa$ be a real positive number. Suppose that $p$ is a large prime number with $\kappa<p^{(1+\varepsilon)\sigma}$ for certain $\varepsilon>0$, then for any integer $m\geq 0$, we have 
		\[M_p^{(m)}(\kappa)=p^{-m\sigma}f^{(m)}\left(\frac{\kappa}{p^\sigma}\right)+O\left(\frac{1}{p^{(m+1-\varepsilon)\sigma}}\right).\]
		When $m=0$, we have 
		\[M_p(\kappa)=f\left(\frac{\kappa}{p^\sigma}\right)+O\left(\frac{\kappa}{p^{2\sigma}}\right),\]
		where all the implicit constants depend only on $\sigma$ and $\epsilon$. 
	\end{lem}
	
	\begin{proof}
		Suppose first $m\geq 1$. Since $\kappa<p^{(1+\epsilon)\sigma}$, we have 
		\begin{align*}
			\mbe^{(m)}\left(|1-X(p)p^{-\sigma}|^{-\kappa}\right)&=\frac{1}{2\pi}\int_0^{2\pi}|1-\mme^{\mmi \theta}p^{-\sigma}|^{-\kappa}\log^m |1-\mme^{\mmi \theta}p^{-\sigma}|^{-1}\mmd \theta\\
			&=\frac{1}{2\pi}\int_0^{2\pi}\left(1-\frac{2\cos\theta}{p^{\sigma}}+\frac{1}{p^{2\sigma}}\right)^{-\frac{\kappa}{2}}(-2)^{-m}\log^m\left(1-\frac{2\cos\theta}{p^{\sigma}}+\frac{1}{p^{2\sigma}}\right)\mmd\theta\\
			&=\frac{1}{2\pi}\int_0^{2\pi}\exp\left(\frac{\kappa\cos\theta}{p^{\sigma}}+O\Big(\frac{\kappa}{p^{2\sigma}}\Big)\right)\left(\frac{\cos^m\theta}{p^{m\sigma}}+O\Big(\frac{1}{p^{(m+1)\sigma}}\Big)\right)\mmd \theta\\
			&=\frac{1}{2\pi}\int_0^{2\pi}\exp\left(\frac{\kappa\cos\theta}{p^{\sigma}}\right)\frac{\cos^m\theta}{p^{m\sigma}}\,\mmd \theta\left(1+O\Big(\frac{1}{p^{(m+1-\varepsilon)\sigma}}\Big)\right)\\
			&=p^{-m\sigma}I_0^{(m)}\left(\frac{\kappa}{p^{\sigma}}\right)\left(1+O\Big(\frac{1}{p^{(m+1-\varepsilon)\sigma}}\Big)\right).
		\end{align*}
		Therefore
		$$\gamma_{m}\coloneqq \frac{1}{m!}\frac{\mbe^{(m)}}{\mbe}=\frac{1}{m!}\frac{f^{(m)}(\kappa/p^{\sigma})}{f(\kappa/p^{\sigma})}+O\left(\frac{1}{p^{(m+1-\varepsilon)}}\right).$$
		Then by applying Faà di Bruno's formula, we have 
		\begin{align*}
			M_p^{(m)}(\kappa) &=\sum_{\substack{i_1\geq 0,\dots ,i_N\geq 0\\1i_1+\dots+Ni_N=m}} (-1)^{i_1+\dots+i_N-1} m!(i_1+\dots+i_N-1)!\frac{\gamma_{1}^{i_1}\dots\gamma_{N}^{i_N}}{i_1!\dots i_N!}\\
			&=p^{-m\sigma}\left(\log I_0\right)^{(m)}\left(\frac{\kappa}{p^\sigma}\right)+O\left(\frac{1}{p^{(m+1-\varepsilon)\sigma}}\right).
		\end{align*}
		The evaluation of $M_p$ in the case $m=0$ is similar.
	\end{proof}
	
	When $p$ is smaller, for every single $M_p^{(m)}$, the asymptotic expansion can be deduced from Watson's lemma. By definition, we have 
	\begin{align*}
		\mbe^{(m)}\left(|1-X(p)p^{-\sigma}|^{-\kappa}\right)&=\frac{1}{2\pi}\int_0^{2\pi}|1-\mme^{\mmi \theta}p^{-\sigma}|^{-\kappa}\log^m|1-\mme^{\mmi \theta}p^{-\sigma}|^{-1}\mmd \theta\\
		&=\frac{(-1)^m}{2^m\pi}\int_0^{\frac{\pi}{2}}\mme^{-\frac{\kappa}{2}g(\theta)}g^m(\theta)\mmd\theta+O\left(\mme^{-\kappa^{\varepsilon}}\right),
	\end{align*}
	where $g(\theta)=\log\,(1-\frac{2\cos\theta}{p^\sigma}+\frac{1}{p^{2\sigma}}/2)$. Let $G(\theta)=g(\theta)-g(0)$. It is easy to see that $G(\theta)$ is monotonically increasing in the interval $[0,\frac{\pi}{2}]$ and reaches its maximal value $A=G(\frac{\pi}{2})$ at $\theta=\frac{\pi}{2}$.  Then for $t\in [0,A(a)]$ we have 
	\[\varphi(t):=\frac{1}{G'(G\inv(t))}=\frac{\mme^{\frac{t}{2}}\sinh\frac{a}{2}}{\sqrt{\cosh a-\cosh a\cosh t+\sinh t}},\]
	where $a=\sigma\log p\in (\frac{1}{2}\log 2,\infty)$ and $A(a)=\log \cosh a -\log (\cosh a-1)$. The function $\varphi(t)$ has the following expansion at $t=0$
	\begin{align*}
		\varphi(t)&=\sinh\frac{a}{2}\, t^{-\frac{1}{2}}+\frac{1}{4}\left(2+\cosh a\right)\sinh\frac{a}{2}\,t^{\frac{1}{2}}+\frac{1}{96}\left(2+3\cosh a\right)^2\sinh\frac{a}{2}\,t^{\frac{3}{2}}\\
		&+\frac{1}{384}\left(15\cosh^3 a+18 \cosh^2 a-4 \cosh a-8\right)\sinh\frac{a}{2}\,t^{\frac{5}{2}}+O\big(t^{\frac{7}{2}}\big).  
	\end{align*}
	
	To give a uniform estimation for $\mbe^{(m)}$ for $p$, one must seek for a global asymptotic expansion. So we need an elaborated estimation on the function $\varphi$.  
	\begin{lem}
		There exists $M>0$ such that the following uniform bound holds 
		\begin{align*}\label{eqn:phi}
			\Big|\varphi(t)-\sinh\frac{a}{2}\, t^{-\frac{1}{2}}-\frac{1}{4}(2+&\cosh a)\sinh\frac{a}{2}\,t^{\frac{1}{2}}-\frac{1}{96}\left(2+3\cosh a\right)^2\sinh\frac{a}{2}\,t^{\frac{3}{2}}\Big|\\
			&\leq \frac{M}{384}\left(15\cosh^3 a+18 \cosh^2 a-4 \cosh a-8\right)\sinh\frac{a}{2}\,t^{\frac{5}{2}},
		\end{align*}
		for all $a\in[\frac{1}{2}\log 2,\infty)$ and $t\in [0,A(a)]$.  
	\end{lem}
	\begin{proof}
		The function $g(t)=\csch\frac{a}{2}\,t^{\frac{1}{2}}\varphi(t)$ is analytic in the disk $\{t\in\mbc\,\big|\,|t|<R(a)\}$ where
		$$R(a)=\log\left(\frac{\cosh a+1}{\cosh a-1}\right)\in \big(0,\log\big(17+12\sqrt{2}\big)\big].$$
		The function $A(a)/R(a)$ is monotone decreasing with respect to $a$. It attains its maximal value $\rho=0.811\dots$ when $a=\frac{1}{2}\log 2$. The Cauchy bound of $g(t)$ for $|t|<A(a)$ gives
		$$\big|g^{(3)}(t)\big|\leq \frac{3!\, 2\pi}{(\rho_0-\rho)^{4}}R(a)^{-3}\max_{|t|=\rho_0 R}\big|g(t)\big|.$$
		Take $\rho_0=\frac{9}{10}$ and  we get $\max_{|t|=\rho_0 R}|g(t)|=|g(\rho_0R(a))|<17$ for all $a$. This implies
		$$\big|g^{(3)}(t)\big|\leq 2\times10^{7}R(a)^{-3}\asymp e^{3a}$$
		as $a\to\infty$.
		On the other hand,
		$$\big|g^{(3)}(0)\big|=\frac{1}{384}\left(15\cosh^3 a+18 \cosh^2 a-4 \cosh a-8\right)\asymp e^{3a}.$$
		Thus there exists certain $M>0$ such that $|g^{(3)}(t)| \leq M|g^{(3)}(0)|$.
	\end{proof}
	
	With the above estimations, we are now able to prove
	\begin{lem}\label{lem:Mp01}
		Let $\kappa$ be a large real number. For any prime number $p$ with $\kappa>p^{(1+\varepsilon)\sigma}$ for some certain $\varepsilon>0$, we have
		\begin{align*}
			M_{p}(\kappa)&=-\kappa\log (1-p^{-\sigma})+O(\log \kappa),\\
			M_{p}'(\kappa)&=-\log (1-p^{-\sigma})\left(1+O(\kappa^{-\varepsilon})\right),
		\end{align*}
		where the implicit constants depend only on $\sigma$ and $\epsilon$.
	\end{lem}
	\begin{proof}
		We have 
		\begin{align*}
			\mbe^{(m)}\left(|1-X(p)p^{-\sigma}|^{-\kappa}\right)
			&=\frac{(-1)^m}{2^m\pi}\int_0^{\frac{\pi}{2}}\mme^{-\frac{\kappa}{2}g(\theta)}g^m(\theta)\mmd \theta+O\left(\mme^{-\kappa^{\varepsilon}}\right)\\
			&=\frac{(-1)^m}{2^m\pi}\int_0^{\frac{\pi}{2}}\mme^{-\frac{\kappa}{2}(G(\theta)+g(0))}(G(\theta)+g(0))^m\mmd \theta+O\left(\mme^{-\kappa^{\varepsilon}}\right)\\
			&=\frac{(-1)^m}{2^m\pi \kappa}\mme^{-\frac{\kappa}{2}g(0)}\int_0^{\frac{1}{2}\kappa A}\mme^{-t}\left(\frac{2t}{\kappa}+g(0)\right)^m\varphi\left(\frac{2t}{\kappa}\right)\mmd t+O\left(\mme^{-\kappa^{\varepsilon}}\right).
		\end{align*}
		
		If $m=0$, since $\cosh a=\frac{p^{\sigma}+p^{-\sigma}}{2}\ll \kappa^{1-\epsilon}$, we have
		\[\mbe\left(|1-X(p)p^{-\sigma}|^{-\kappa}\right)=\frac{1}{\pi \kappa}\mme^{-\frac{\kappa}{2}g(0)}\sinh\frac{a}{2}\left(\left(\frac{2}{\kappa}\right)^{-\frac{1}{2}}\,\gamma\left(\frac{1}{2},\frac{\kappa A}{2}\right)+O\left(\kappa^{-\frac{1}{2}-\epsilon}\right)\right),\]
		where the implicit constant depends only on $\sigma$ and $\varepsilon$.
		Since 
		\[A=\log\frac{1+p^{-\sigma}}{1-p^{-\sigma}}=2p^{-\sigma}+O(p^{-2\sigma})\gg \kappa^{-1-\varepsilon},\]
		so we have $\kappa A\gg \kappa^{\epsilon}$ and 
		\[\gamma\left(\frac{1}{2},\frac{\kappa A}{2}\right)=\Gamma\left(\fot\right)+O\left(\mme^{-\kappa^{\varepsilon}}\right).\]
		Hence
		$$\mbe\left(|1-X(p)p^{-\sigma}|^{-\kappa}\right)=\frac{1}{\sqrt{2\pi}}\sinh \frac{a}{2}\mme^{-\kappa\log(1-p^{-\sigma})}\kappa^{-\frac{1}{2}}\left(1+O(\kappa^{-\varepsilon})\right)$$
		Take logarithm and we get
		\begin{equation}\label{eq:e0xp}
			M_p(\kappa)=-\kappa\log (1-p^{-\sigma})+O(\log \kappa).
		\end{equation}
		
		If $m=1$, in the same manner as Eq. \eqref{eq:e0xp} one has
		\begin{equation}\label{eq:e1xp}
			\mbe'\left(|1-X(p)p^{-\sigma}|^{-\kappa}\right)=-\frac{1}{\sqrt{2\pi}}\sinh \frac{a}{2}\mme^{-\kappa\log(1-p^{-\sigma})}\log(1-p^{-\sigma})\kappa^{-\frac{1}{2}}\left(1+O(\kappa^{-\varepsilon})\right).
		\end{equation}
		This yields
		\[
		M_p'(\kappa)=\frac{\mbe'}{\mbe}=-\log (1-p^{-\sigma})\left(1+O(\kappa^{-\varepsilon})\right).\qedhere
		\]
	\end{proof}
	
	Uniform asymptotic behaviors for higher order derivatives can be also obtained for all smaller $p$ and for all $z$ in a fixed sector.
	\begin{lem}\label{lem:Mp2}
		Let $z\in\C$. Suppose $\arg z\in (-\theta, \theta)$ for certain fixed $0<\theta<\frac{\pi}{2}$. For any prime number $p$ with $|z|>p^{(2+\varepsilon)\sigma}$ for some certain $\varepsilon>0$, we have
		$$ M_{p}^{(m)}(z)=(-1)^m\frac{(m-1)!}{2}z^{-m}+O(z^{-m-\varepsilon}),$$
		where the implicit constant depends only on $\sigma$, $\varepsilon$  and $\theta$.
	\end{lem}
	\begin{proof}
		In the same way as Lemma~\ref{lem:Mp01} but using $\cosh a=\frac{p^{\sigma}+p^{-\sigma}}{2}\ll |z|^{\frac{1}{2}-\epsilon}$, one gets
		\begin{multline*}
			\mbe''\left(|1-X(p)p^{-\sigma}|^{-z}\right)=\frac{1}{\sqrt{2\pi}}\sinh \frac{a}{2}\mme^{-\lambda_{p}z}\lambda_{p}^{2}\\
			\cdot\left(z^{-\frac{1}{2}}+\frac{1}{4}\left(4+2\lambda_{p}+\lambda_{p}\cosh a\right)\lambda_{p}\,z^{-\frac{3}{2}}+O_{\sigma,\varepsilon,\theta}\left(z^{-2-\varepsilon}\right)\right),
		\end{multline*}
		where $\lambda_p=\log (1-p^{-\sigma})$. Similar asymptotic expansions can also be obtained for $\mbe$ and $\mbe'$. At last, one can get
		$$M_p''(z)=\frac{\mbe''}{\mbe}-\frac{\mbe'^2}{\mbe^2}=\frac{1}{2}z^{-2}+O_{\sigma,\varepsilon,\theta}(z^{-2-\varepsilon}).$$
		Taking derivatives with respect to $z$ and using Cauchy's estimate, we get the desired result for $n>2$.
	\end{proof}
	
	With the above preparation, we are able to determine the asymptotic behaviors of $M(\kappa)$ and its derivatives.
	\begin{prop}\label{prop:saddle1} 
		As the positive real numbers $\kappa\to\infty$, we have
		\begin{equation}\label{Moments1}
			M(\kappa)= \frac{\kappa^{1/\sigma}}{\log \kappa}\left(a_0^{(0)}+\frac{a_1^{(0)}}{\log \kappa}+\dots+\frac{a_N^{(0)}}{(\log \kappa)^N}+O\left(\frac{1}{(\log \kappa)^{N+1}}\right)\right),
		\end{equation}
		where $$ a_n^{(0)}:=\int_0^{\infty} \frac{f(u)(\log u)^n}{u^{1/\sigma+1}}\mmd u,$$
		and
		\begin{equation}\label{Moments2}
			M'(\kappa)= \frac{\kappa^{1/\sigma-1}}{\log \kappa}\left(a_0^{(1)}+\frac{a_1^{(1)}}{\log \kappa}+\dots+\frac{a_N^{(1)}}{(\log \kappa)^N}+O\left(\frac{1}{(\log \kappa)^{N+1}}\right)\right).
		\end{equation}
		where 
		$$ a_n^{(1)}:=\int_0^{\infty} \frac{f'(u)(\log u)^n}{u^{1/\sigma}}\mmd u.$$ 
	\end{prop}
	\begin{proof}
		
		So by combining Lemma \ref{lem:pzsigma}, Lemma~\ref{lem:Mp01} and Prime Number Theorem, we have 
		\begin{equation}\label{eq:mzht}
			\begin{split}
				M(\kappa)+O(\kappa\log \kappa)&=\sum_{p^{\sigma}<\kappa^{1/(1+\varepsilon)}}M_{p}(\kappa)+\sum_{p^{\sigma}>\kappa^{1/(1+\varepsilon)}}M_{p}(\kappa)\\
				&=\sum_{p^{\sigma}<\kappa^{1/(1+\varepsilon)}}-\log(1-p^{-\sigma})\,\kappa+\sum_{p^{\sigma}>\kappa^{1/(1+\varepsilon)}}f(p^{-\sigma}\kappa)\\
				&=-\kappa\int_{\kappa^{-1/(1+\varepsilon)}}^{2^{-\sigma}}\frac{\log(1-t)}{\log t}t^{-\frac{1}{\sigma}-1}\mmd t-\int_0^{\kappa^{-1/(1+\varepsilon)}}\frac{f(t\kappa)}{\log t}t^{-\frac{1}{\sigma}-1}\mmd t\\
				&=-\kappa^{\frac{1}{\sigma}}\int_{\kappa^{\varepsilon/(1+\varepsilon)}}^{2^{-\sigma}\kappa}\frac{\kappa\log(1-t/\kappa)}{\log t-\log \kappa}t^{-\frac{1}{\sigma}-1}\mmd t-\kappa^{\frac{1}{\sigma}}\int_0^{\kappa^{\varepsilon/(1+\varepsilon)}}\frac{f(t)}{\log t-\log \kappa}t^{-\frac{1}{\sigma}-1}\mmd t
			\end{split}
		\end{equation}
		When $0\leq t\leq 2/\kappa$, by Lemma \ref{lem:Bessel2}, we have $f(t)\asymp t^2$, so 
		\[\int_0^{2/\kappa}\frac{f(t)}{\log \kappa-\log t}t^{-\frac{1}{\sigma}-1}\mmd t\asymp \int_0^{2/\kappa}\frac{t^{-\frac{1}{\sigma}+1}}{\log \kappa}\mmd t\asymp\frac{1}{z^{1-\frac{\sigma}{2}}\log \kappa}.\]
		Then a direct expansion with respect to $\log t$ gives  
		\begin{equation}\label{eq:mzhtn}
			M(\kappa)=\frac{\kappa^{1/\sigma}}{\log \kappa}\left(\sum_{n\leq N}\frac{1}{(\log \kappa)^n}\int_0^\infty \frac{h(t)(\log t)^n}{t^{\frac{1}{\sigma}}}\mmd t+O((\log \kappa)^{-N-1})\right),
		\end{equation}
		where 
		\[h(t)=\begin{cases} 1& t\geq \kappa^{\varepsilon/(1+\varepsilon)}\\ t\inv f(t) & t\leq \kappa^{\varepsilon/(1+\varepsilon)}.\end{cases}.\]
		We note that $t\inv f(t)-1=o(1)$ as $t\to \infty$, so 
		\begin{align*}
			&\int_0^\infty \frac{h(t)(\log t)^n}{t^{1/\sigma}}\mmd t=\int_0^\infty \frac{f(t)(\log t)^n}{t^{1/\sigma+1}}\mmd t+\int_{\kappa^{\varepsilon/(1+\varepsilon)}}^\infty \frac{(t\inv f(t)-1)(\log t)^n}{t^{1/\sigma}}\mmd t\\
			=&\int_0^\infty \frac{f(t)(\log t)^n}{t^{1/\sigma+1}}\mmd t+O(\kappa^{-\varepsilon(1/\sigma-1)}).
		\end{align*}
		This proves Eq. \eqref{Moments1}.
		
		With the same calculation as Eq. \eqref{eq:mzht}, we get 
		\begin{equation}
			M'(\kappa)=\frac{\kappa^{1/\sigma-1}}{\log \kappa}\left(\sum_{n\leq N}\frac{1}{(\log \kappa)^n}\int_0^\infty \frac{h_1(t)(\log t)^n}{t^{1/\sigma}}\mmd t+O((\log \kappa)^{-N-1})\right),
		\end{equation}
		where 
		\[h_1(t)=\begin{cases} 1& t\geq \kappa^{\varepsilon/(1+\varepsilon)}\\ f'(t) & t\leq \kappa^{\varepsilon/(1+\varepsilon)}.\end{cases}.\]
		Finally, we note that $f'(t)-1=t\inv +O(t^{-2})$ as $t\to \infty$, so 
		\begin{align*}
			&\int_0^\infty \frac{h_1(t)(\log t)^n}{t^{1/\sigma}}\mmd t=\int_0^\infty \frac{f'(t)(\log t)^n}{t^{1/\sigma}}\mmd t+\int_{\kappa^{\varepsilon/(1+\varepsilon)}}^\infty \frac{(f'(t)-1)(\log t)^n}{t^{1/\sigma}}\mmd t\\
			=&\int_0^\infty \frac{f'(t)(\log t)^n}{t^{1/\sigma}}\mmd t+O(\kappa^{-\varepsilon/\sigma}).
		\end{align*}
		This completes the proof.
	\end{proof}
	\begin{prop}\label{prop:saddle2}
		As the positive real numbers $\kappa\to\infty$, we have for $m\geq 2$
		\begin{equation}
			M^{(m)}(\kappa)= \frac{\kappa^{1/\sigma-m}}{\log \kappa}\left(a_0^{(m)}+\frac{a_1^{(m)}}{\log \kappa}+\dots+\frac{a_N^{(m)}}{(\log \kappa)^N}+O\left(\frac{1}{(\log \kappa)^{N+1}}\right)\right).
		\end{equation}
		where 
		$$ a_n^{(m)}:=\int_0^{\infty} \frac{f^{(m)}(u)(\log u)^n}{u^{1/\sigma+1-m}}\mmd u.$$ 
	\end{prop}
	\begin{proof}
		Put $\alpha_m=(-1)^m\frac{(m-1)!}{2}$. Using Lemma~\ref{lem:pzsigma} (with $\varepsilon$ replaced by $1+\varepsilon$) and Lemma~\ref{lem:Mp01} again, we have
		\begin{align*}
			M^{(m)}(\kappa)&=\sum_{p^{\sigma}<\kappa^{1/(2+\varepsilon)}}M_{p}^{(m)}(\kappa)+\sum_{p^{\sigma}>\kappa^{1/(2+\varepsilon)}}M_{p}^{(m)}(\kappa)\\
			&=\sum_{p^{\sigma}<\kappa^{1/(2+\varepsilon)}}\alpha_m\,\kappa^{-m}+\sum_{p^{\sigma}>\kappa^{1/(2+\varepsilon)}}p^{-m\sigma}f^{(m)}(p^{-\sigma}\kappa)+O(\kappa^{-m+1+\varepsilon})\\
			&= \alpha_m\,\kappa^{-m}\int_{\kappa^{-1/(2+\varepsilon)}}^{2^{-\sigma}}\frac{1}{\log t}t^{-\frac{1}{\sigma}-1}\mmd t+\int_0^{\kappa^{-1/(2+\varepsilon)}}\frac{f^{(m)}(t\kappa)}{\log t}t^{-\frac{1}{\sigma}+m-1}\mmd t+O(\kappa^{-m+1+\varepsilon})\\
			&=\alpha_m\kappa^{\frac{1}{\sigma}-m}\int_{\kappa^{\varepsilon/(2+\varepsilon)}}^{2^{-\sigma}\kappa}\frac{1}{\log t-\log \kappa}t^{-\frac{1}{\sigma}-1}\mmd t+\kappa^{\frac{1}{\sigma}-m}\int_0^{\kappa^{\varepsilon/(2+\varepsilon)}}\frac{f^{(m)}(t)}{\log t-\log \kappa}t^{-\frac{1}{\sigma}+m-1}\mmd t.
		\end{align*}
		When $0\leq t\leq 2/\kappa$, by Lemma \ref{lem:Bessel2}, we have $f^{(m)}(t)=O(1)$, so 
		\[\int_0^{2/\kappa}\frac{f^{(m)}(t)}{\log t-\log \kappa}t^{-\frac{1}{\sigma}+1}\mmd t\ll \int_0^{2/\kappa}\frac{t^{-\frac{1}{\sigma}+1}}{\log \kappa}\mmd t\ll \frac{1}{\kappa^{2-\frac{1}{\sigma}}\log \kappa}.\]
		Thus
		\begin{equation}\label{eq:mzhtn}
			M^{(m)}(\kappa)=\frac{\kappa^{1/\sigma-m}}{\log \kappa}\left(\sum_{n\leq N}\frac{1}{(\log \kappa)^n}\int_0^\infty \frac{h_m(t)(\log t)^n}{t^{1/\sigma}}\mmd t+O((\log \kappa)^{-N-1})\right),
		\end{equation}
		where 
		\[h_m(t)=\begin{cases} \alpha_m\,t^{-1}& t\geq \kappa^{\varepsilon/(2+\varepsilon)}\\ t^{m-1} f^{(m)}(t) & t\leq \kappa^{\varepsilon/(2+\varepsilon)}\end{cases}.\]
		Again we see that $t^{m-1} f^{(m)}(t)=\alpha_mt^{-1}+o(t^{-2})$ as $t\to \infty$, so 
		\begin{align*}
			&\int_0^\infty \frac{h_m(t)(\log t)^n}{t^{1/\sigma}}\mmd t=\int_0^\infty \frac{f^{(m)}(t)(\log t)^n}{t^{1/\sigma-m+1}}\mmd t+\int_{\kappa^{\varepsilon/(2+\varepsilon)}}^\infty \frac{\big(t^{m-1} f^{(m)}(t)-\alpha_mt^{-1}\big)(\log t)^n}{t^{1/\sigma}}\mmd t\\
			=&\int_0^\infty \frac{f^{(m)}(t)(\log t)^n}{t^{1/\sigma-m+1}}\mmd t+O(\kappa^{-\frac{\varepsilon}{2}(1/\sigma+1)}).
		\end{align*}
		This proves Eq. \eqref{Moments1}.
	\end{proof}
	
	We then apply saddle point method to get the desired asymptotic expansion of $\Psi$. Here we give a precise estimation for the saddle point $\kappa$.
	\begin{lem}\label{lem:kappatau}
		Let $\tau$ be a large real number and $\kappa$ be the unique solution to $M'(k)=\tau$. Then there exist computable polynomials $f_0(\cdot),f_1(\cdot),\cdots,f_N(\cdot)$ with $\deg f_n\le n$ for $0\le n\le N$, which depends only on $\sigma$ and $N$, such that 
		\[\kappa=g(\sigma)(\tau\log \tau)^{\sigma/(1-\sigma)}\left(f_0(\log_2 \tau)+\frac{f_1(\log_2 \tau)}{\log\tau}+\dots+\frac{f_N(\log_2 \tau)}{(\log\tau)^N}+O\bigg(\bigg(\frac{\log_2 \tau}{\log\tau}\bigg)^{\!N+1}\bigg)\right),\]
		where 
		\[g(\sigma)=\left(\frac{1}{a_0^{(1)}}\frac{\sigma}{1-\sigma}\right)^{\frac{\sigma}{1-\sigma}},\]
		and $a_0^{(1)}$ is defined in Proposition \ref{prop:saddle1}. More precisely,
		$$f_0(t)=1,\,\,f_1(t)=\frac{\sigma}{1-\sigma}\,t+\log g(\sigma)-\frac{a_1^{(1)}}{a_0^{(1)}},\,\dots$$
	\end{lem}
	\begin{proof}
		This follows directly from Proposition \ref{prop:saddle1} and Lemma $4.3$ in \cite{Dong22}.
	\end{proof}
	
	\begin{lem}\label{lem:prop7.1}
		Let $1/2\leq  \sigma<1$. Then 
		\[\Psi(\tau)=\frac{\mbe(|L(\sigma,X)|^\kappa)\mme^{-\tau\kappa}}{\kappa\sqrt{2\pi M''(\kappa)}}(1+O(\kappa^{1-1/\sigma}\log \kappa)),\]
		holds uniformly for $\tau\geq 1$.
	\end{lem}
	\begin{proof}
		See \cite[Proposition $7.1$]{LLR19}.
	\end{proof}
	
	\begin{proof}[Proof of Theorem \ref{thm1.1}]
		Combining Proposition \ref{prop:saddle1} and Lemma \ref{lem:prop7.1}, we have 
		\[\Psi(\tau)=\exp\left(\frac{\kappa^{1/\sigma}}{\log \kappa}\left(a_0^{(0)}+\frac{a_1^{(0)}}{\log \kappa}+\dots+\frac{a_N^{(0)}}{(\log \kappa)^N}+O\left(\frac{1}{(\log \kappa)^{N+1}}\right)\right)-\tau\kappa\right).\]
		Then by applying Lemma \ref{lem:kappatau}, we get
		$$\Psi(\tau)=\exp\left(-(\tau\log^{\sigma}\tau)^{\frac{1}{1-\sigma}}\left(\sum_{n=0}^{N}\frac{\mathfrak{a}_n(\log_2\tau)}{(\log\tau)^n}+ O\left(\frac{\log_2\tau}{\log\tau}\right)^{\!N+1}\right)\right),$$
		where
		$$\mfpa_0(t)=g(\sigma)\left(1-\frac{a_0^{(0)}}{a_0^{(1)}}\right),\quad
		\mfpa_1(t)=g(\sigma)\sigma\,t+g(\sigma)(1-\sigma)\left(\log g(\sigma)-\frac{a_1^{(0)}}{a_0^{(0)}}\right),\quad\dots
		$$
		To show $\mathfrak{a}_0>0$, it is equivalent to $a_0^{(0)}<a_0^{(1)}$. But the definition of $a_0^{(0)},a_0^{(1)}$ gives 
		\[a_0^{(1)}=\int_0^\infty u^{-\frac{1}{\sigma}}\mmd f(u)=\frac{1}{\sigma}\int_0^\infty \frac{f(u)}{u^{1+\frac{1}{\sigma}}}\mmd u=\frac{a_0^{(0)}}{\sigma}>a_0^{(0)}.\]
		Here the second integral converges absolutely from the facts that $f(u)\asymp u^2$ as $u\to 0$ and $f(u)\asymp u$ as $u\to \infty$.
	\end{proof}

	\subsection{\blue{When $\sigma=1$}}

	In this subsection, we switch to the case $\sigma=1$. We will see the computation in this case is similar to the case $\frac{1}{2}<\sigma<1$ with a few modification. Firstly, we give asymptotic formulas for the function $M(\kappa)$ and its derivatives, like we did in Proposition~\ref{prop:saddle1} and Proposition~\ref{prop:saddle2}
	\begin{prop}
		As the positive real numbers $\kappa\to\infty$, we have
		$$M(\kappa)=-\kappa\,\sum_{p<\kappa}\log(1-p\inv)+\frac{\kappa}{\log \kappa}\bigg(\sum_{n\le N}\frac{A_n^{(0)}}{(\log \kappa)^n}+O\Big(\frac{1}{(\log \kappa)^{N+1}}\Big)\bigg),$$
		where \begin{align}
			A_n^{(0)}=\int_0^1(\log t)^n\frac{f(t)}{t^2}\mmd t+\int_1^\infty(\log t)^n\frac{f(t)-t}{t^2}\mmd t,\label{defofA_n}
		\end{align}
		and
		$$M'(\kappa)=-\,\sum_{p<\kappa}\log(1-p\inv)+\frac{1}{\log \kappa}\bigg(\sum_{n\le N}\frac{A_n^{(1)}}{(\log \kappa)^n}+O\Big(\frac{1}{(\log \kappa)^{N+1}}\Big)\bigg),$$
		where $$A_n^{(1)}=\int_0^1(\log t)^n\frac{f'(t)}{t}\mmd t+\int_1^\infty(\log t)^n\frac{f'(t)-1}{t}\mmd t.$$
	\end{prop}
	\begin{proof}
		Together with prime number theorem, Lemma \ref{lem:Mp01} gives 
		\[\sum_{p<\kappa^{1/(1+\varepsilon)}}M_p(\kappa)=-\kappa\sum_{p<\kappa^{1/(1+\varepsilon)}}\log(1-p\inv)+O(\kappa^{1/(1+\varepsilon)}).\]
		A similar calculation as Eq. \eqref{eq:mzht} gives 
		\[\sum_{p>\kappa^{1/(1+\varepsilon)}}M_p(\kappa)=\kappa\int_0^{\kappa^{\varepsilon/(1+\varepsilon)}}\frac{f(t)}{t^2}\frac{\mmd t}{\log \kappa-\log t}+O(\kappa^{\varepsilon}).\]
		Following the Mertens' formula \cite[Section $1.6$, Theorem $11$]{Te95}, we have 
		\begin{align*}
			\sum_{\kappa^{1/(1+\varepsilon)}<p<\kappa}\log(1-p\inv)\inv&=\log\log \kappa-\log_2 \kappa^{1/(1+\varepsilon)}+O\left(\mme^{-\sqrt{\log \kappa}}\right)\\
			&=\log(1+\varepsilon)+O\left(\mme^{-\sqrt{\log \kappa}}\right).
		\end{align*}
		On the other hand, we have
		\[\int_1^{\kappa^{\varepsilon/(1+\varepsilon)}}\frac{1}{t}\frac{\mmd t}{(\log \kappa-\log t)}=-\log(\log \kappa-\log t)\big|_1^{\kappa^{\varepsilon/(1+\varepsilon)}}=\log(1+\varepsilon).\]
		So we have 
		\begin{align*}  &M(\kappa)=\sum_{p<\kappa^{1/(1+\varepsilon)}}M_p(\kappa)+\sum_{p>\kappa^{1/(1+\varepsilon)}}M_p(\kappa)\\
			&=-\kappa\,\sum_{p<\kappa}\log(1-p\inv)-\kappa\int_1^{\kappa^{\varepsilon/(1+\varepsilon)}}\frac{\mmd t}{t(\log \kappa-\log t)}+\kappa\int_0^{\kappa^{\varepsilon/(1+\varepsilon)}}\frac{f(t)}{t^2}\frac{\mmd t}{\log\kappa-\log t}\\
			& =-\kappa\,\sum_{p<\kappa}\log(1-p\inv)+\frac{\kappa}{\log \kappa}\,\sum_{n=0}^{\infty}(\log \kappa)^n\left(\int_{0}^{1}(\log t)^n\frac{f(t)}{t^2}\mmd t+\int_{1}^{\kappa^{\varepsilon/(1+\varepsilon)}}(\log t)^n\frac{f(t)-t}{t^2}\mmd t\right).
		\end{align*}
		Meanwhile, from $f(t)-t=O(\log t)$ as $t\to\infty$ we see
		$$
		\int_{\kappa^{\varepsilon/(1+\varepsilon)}}^{\infty}(\log t)^n\frac{f(t)-t}{t^2}\mmd t=O(\kappa^{-\varepsilon}).
		$$
		We conclude that
		$$M(\kappa)=-\kappa\,\sum_{p<\kappa}\log(1-p\inv)+\frac{\kappa}{\log \kappa}\bigg(\sum_{n\le N}\frac{A_n^{(0)}}{(\log \kappa)^n}+O\Big(\frac{1}{(\log \kappa)^{N+1}}\Big)\bigg).$$
		Along the same line, we can get the desired asymptotic expansion for $M'(\kappa)$.
	\end{proof}
	
	\begin{prop}\label{prop4.3}
		As the positive real number $\kappa\to\infty$, we have for $m\geq 2$
		\begin{equation}
			M^{(m)}(\kappa)= \frac{\kappa^{1-m}}{\log \kappa}\left(A_0^{(m)}+\frac{A_1^{(m)}}{\log \kappa}+\dots+\frac{A_N^{(m)}}{(\log \kappa)^N}+O\left(\frac{1}{(\log \kappa)^{N+1}}\right)\right),  
		\end{equation}
		where 
		$$A_n^{(m)}:=\int_0^{\infty} f^{(m)}(u)u^{m-2}(\log u)^n \mmd u.$$
	\end{prop}
	\begin{proof}
		The proof here is exactly the same as the proof of Proposition~\ref{prop:saddle2} where we replace $\sigma$ by 1.
	\end{proof}

	\begin{lem}\label{SmoothPerron}
		Let $\lambda>0$ be a real number and $N$ be a positive integer. For any $c>0$ we have for $y>0$
		$$
		0\leq \frac{1}{2\pi \mmi}\int_{c-\mmi\infty}^{c+\mmi\infty} y^s \frac{\mme^{\lambda s}-1}{\lambda s}\frac{\mmd s}{s} -\textbf{1}_{\ge1}(y)\leq 
		\frac{1}{2\pi \mmi}\int_{c-\mmi\infty}^{c+\mmi\infty} y^s \frac{\mme^{\lambda s}-1}{\lambda s} \frac{1-\mme^{-\lambda  s}}{s}\mmd s.
		$$
	\end{lem}
	\begin{lem}\label{Decay} Let $s=\kappa+\mmi t$ where $\kappa$ is a large positive real number. Then, in the range $|t|\geq \kappa$ we have 
		$$\mbe\left(|L(1, X)|^s\right)\ll \exp\left(-\sqrt{|t|}\right) \mbe\left(|L(1, X)|^\kappa\right).$$
	\end{lem}
	Let us recall that
	$$\Psi_1(\tau):={\rm Prob}(|L(1,X)|>{\mme}^\gamma\tau).$$
	Now using the saddle point method, we can estimate the function $\Psi_1$ with the previous results on $M(\kappa)$.
	\begin{prop}\label{asyPsi1}
		Uniformly for $\tau\ge1$, let $\kappa$ be a positive number such that $\log\tau+\gamma=M'(\kappa)$. Then we have
		\[ \Psi_1(\tau)=\frac{{\mathbb E}\left(|L(1,X)|^{\kappa}\right)(\mme^\gamma\tau)^{-\kappa}}{\kappa\sqrt{2\pi M''(\kappa)}}
		\left(1+ O\left( \frac{\log\kappa}{\sqrt\kappa}\right)\right)\]
	\end{prop}
	\begin{proof}
		Let $0<\lambda<1/(2\kappa)$ be a real number to be chosen later. Using Lemma \ref{SmoothPerron}  we obtain
		\begin{equation}\label{approximation1}
			\begin{aligned}
				0&\leq \frac{1}{2\pi \mmi}\int_{\kappa-\mmi\infty}^{\kappa+\mmi\infty}{\mathbb E}\left(|L(1,X)|^s\right)(\mme^\gamma\tau)^{- s}\frac{\mme^{\lambda s}-1}{\lambda s}\frac{\mmd s}{s}-\Psi(\tau)\\
				&\leq \frac{1}{2\pi \mmi}\int_{\kappa-\mmi\infty}^{\kappa+\mmi\infty} {\mathbb E}\left(|L(1,X)|^s\right)(\mme^\gamma\tau)^{- s} \frac{\mme^{\lambda s}-1}{\lambda s} \frac{1-\mme^{-\lambda s}}{s}\mmd s.
			\end{aligned}
		\end{equation}
		Since $\lambda\kappa<1/2$ we have
		$|\mme^{\lambda s}-1|\leq 3 \text{ and } |\mme^{-\lambda s}-1|\leq 2$. 
		Therefore, using Lemma \ref{Decay} we obtain
		\begin{equation}\label{error12}
			\bigg(\int_{\kappa-\mmi\infty}^{\kappa-\mmi\kappa}+ \int_{\kappa+\mmi\kappa}^{\kappa+\mmi\infty}\bigg){\mathbb E}\left(|L(1,X)|^s\right)(\mme^\gamma\tau)^{- s}\frac{\mme^{\lambda s}-1}{\lambda s}\frac{\mmd s}{s} \ll \frac{\mme^{-\sqrt\kappa}}{\lambda \kappa} {\mathbb E}\left(|L(1,X)|^{\kappa}\right)\tau^{-\kappa},
		\end{equation}
		and similarly
		\begin{equation}\label{error2}
			\bigg(\int_{\kappa-\mmi\infty}^{\kappa-\mmi\kappa}+ \int_{\kappa+\mmi\kappa}^{\kappa+\mmi\infty}\bigg){\mathbb E}\left(|L(1,X)|^s\right)(\mme^\gamma\tau)^{- s} \frac{\mme^{\lambda s}-1}{\lambda s} \frac{1-\mme^{-\lambda s}}{s}\mmd s \ll \frac{\mme^{-{\sqrt\kappa}}}{\lambda \kappa} {\mathbb E}\left(|L(1,X)|^{\kappa}\right)\tau^{-\kappa}.
		\end{equation}
		Furthermore, if $|t|\leq \kappa$ then $\left|(1-\mme^{-\lambda s})(\mme^{\lambda s}-1)\right|\ll \lambda^2|s|^2$. Hence we derive 
		$$
		\int_{\kappa-\mmi\kappa}^{\kappa+\mmi\kappa} \mbe\left(|L(1,X)|^s\right)(\mme^\gamma\tau)^{- s} \frac{\mme^{\lambda s}-1}{\lambda s} \frac{1-\mme^{-\lambda s}}{s}\mmd s \ll \lambda\kappa{\mathbb E}\left(|L(1,X)|^{\kappa}\right)\tau^{- \kappa}.
		$$ 
		Therefore, combining this estimate with equations \eqref{approximation1}, \eqref{error12} and \eqref{error2} we deduce that
		\begin{equation}\label{approximation2}
			\begin{aligned}
				\Psi(\tau) &- \frac{1}{2\pi \mmi}\int_{\kappa-\mmi\kappa}^{\kappa+\mmi\kappa}{\mathbb E}\left(|L(1,X)|^s\right)(\mme^\gamma\tau)^{- s}\frac{\mme^{\lambda s}-1}{\lambda s^2} \mmd s \\
				&\ll \left(\lambda\kappa+\frac{\mme^{-\sqrt\kappa}}{\lambda \kappa}\right){\mathbb E}\left(|L(1,X)|^{\kappa}\right)\tau^{-\kappa}.\\
			\end{aligned}
		\end{equation}
		On the other hand, in the region $|t|\leq \kappa$ we have 
		\[\log{\mathbb E}\left(|L(1, X)|^{\kappa+\mmi t}\right)= \log{\mathbb E}\left(|L(1,X)|^{\kappa}\right)+\mmi tM'(\kappa)-\frac{t^2}{2}M''(k)+ O\left(|M'''(\kappa)||t|^{3}\right).\]
		Also, note that
		$$ \frac{\mme^{\lambda s}-1}{\lambda s^2}= \frac{1}{\kappa}\left(1-\mmi\frac{t}{\kappa}+ O\left(\lambda \kappa+\frac{t^2}{\kappa^2}\right)\right).$$
		Hence, using that $M'(\kappa)=\log\tau+\gamma$ we obtain 
		\begin{align*}
			&{\mathbb E}\left(|L(1,X)|^s\right)(\mme^\gamma\tau)^{- s}\frac{\mme^{\lambda s}-1}{\lambda s^2}\\
			= &\frac{1}{\kappa}{\mathbb E}\left(|L(1,X)|^{\kappa}\right)\mme^{-\gamma\kappa}\tau^{-\kappa}\exp\left(-\frac{t^2}{2}M''(\kappa)\right) 
			\left(1-\mmi\frac{t}{\kappa}+O\left(\lambda\kappa+ \frac{t^2}{\kappa^2}+ |M'''(\kappa)||t|^3\right)\right)
		\end{align*}
		Therefore, we obtain
		\begin{align*}
			&\frac{1}{2\pi \mmi}\int_{\kappa-\mmi\kappa}^{\kappa+\mmi\kappa}{\mathbb E}\left(L|(1,X)|^s\right)(\mme^\gamma\tau)^{- s}\frac{\mme^{\lambda s}-1}{\lambda s^2}\mmd s\\
			=& \frac{1}{\kappa}{\mathbb E}\left(|L(1,X)|^{\kappa}\right)(\mme^\gamma\tau)^{-\kappa} \frac{1}{2\pi} \int_{-\kappa}^{\kappa}\exp\left(-\frac{t^2}{2}M''(\kappa)\right)
			\left(1+ O\left(\lambda\kappa+ \frac{t^2}{\kappa^2}+ |M'''(\kappa)||t|^3\right)\right) \mmd t,
		\end{align*}
		since the integral involving $it/{\kappa}$ vanishes. Further, we have 
		$$ 
		\frac{1}{2\pi} \int_{-\kappa}^{\kappa}\exp\left(-\frac{t^2}{2}M''(\kappa)\right)\mmd t= \frac{1}{\sqrt{2\pi M''(\kappa)}}\left(1+O\left(\exp\left(-\frac12\kappa^2
		M''(\kappa)\right)\right)\right),
		$$
		and 
		$$ \int_{-\kappa}^{\kappa}|t|^n\exp\left(-\frac{t^2}{2}M''(\kappa)\right)\mmd t\ll \frac{1}{M''(\kappa)^{(n+1)/2}}.$$
		Thus, using Proposition \ref{prop4.3} for the order of $M''(\kappa)$ and $M'''(\kappa)$ we deduce that
		\begin{equation}\label{main}
			\begin{aligned}
				&\frac{1}{2\pi \mmi}\int_{\kappa-\mmi\kappa}^{\kappa+\mmi\kappa}{\mathbb E}\left(|L(1,X)|^s\right)(\mme^\gamma\tau)^{- s}\frac{\mme^{\lambda s}-1}{\lambda s^2} \mmd s\\
				=& \frac{{\mathbb E}\left(|L(1,X)|^{\kappa}\right)(\mme^\gamma\tau)^{-\kappa}}{\kappa\sqrt{2\pi M''(\kappa)}}
				\left(1+ O\left(\lambda\kappa+ \frac{\log \kappa}{\sqrt\kappa}\right)\right).
			\end{aligned}
		\end{equation}
		Thus, combining the estimates \eqref{approximation2} and \eqref{main} and choosing $\lambda= \kappa^{-3}$ completes the proof.
	\end{proof}
	As a direct result of Proposition \ref{asyPsi1}, we have the following asymptotic formula for the distribution function $\Psi_1(\tau)$, which is the second part of Theorem \ref{thm1.2}.
	\begin{cor}\label{coro4.6}
		There is a sequence of real numbers $\{\mathfrak{b}_n\}_{n\ge 1}$ such that for any integer $N\ge 1$ we have
		\begin{equation*}
			\Psi_1(\tau)=\exp\bigg(\!-\frac{{\rm e}^{\tau-A_0-1}}{\tau}
			\bigg\{1+\sum_{n=1}^N \frac{\mathfrak{b}_n}{\tau^n}
			+ O_N\bigg(\frac{1}{\tau^{N+1}}\bigg)\bigg\}\bigg)
		\end{equation*}
		uniformly for  $  \tau\ge2$,
		where $A_0=A_0^{(0)}$ is defined in \eqref{defofA_n}.
	\end{cor}

	\section{\blue{Large Deviations: Proof of Theorem \ref{thm1.2}}}\label{chap4}
	\subsection{\blue{When $\frac12<\sigma<1$}}
	The following proposition is a key to the proof of the first part of Theorem \ref{thm1.2}. It shows that, after removing a small set of ``bad'' characters, the moment of Dirichlet $L$-functions can be computed well by the random models.
	\begin{prop}\label{prop4.1}
		Let $\frac12<\sigma<1$ and $A>1$ be fixed. There exist positive constants $b_3=b_3(\sigma,A)$ and $b_4=b_4(\sigma,A)$ and a set ${\mathcal E}(q)\subseteq\{\chi:\;\chi({\rm mod}\;q)\}$ with $\#{\mathcal E}(q)\leq q\exp(-b_3\log q/\log_2q)$, such that for all complex numbers $z$ with $|z|\le  b_4(\log q)^\sigma$ we have 
		\[\frac{1}{q}\sum_{\chi({\rm mod}\;q)\atop \chi\notin{\mathcal E}(q)}|L(\sigma,\chi)|^z={\mathbb E}(|L(\sigma,X)|^z)+O\bigg(\frac{{\mathbb E}(|L(\sigma,X)|^{{\rm Re\,}z})}{(\log q)^{A-\sigma}}\bigg).\]
	\end{prop}
	
	\begin{proof}
		Let $\mce(q)={\mathcal A}_q\cup{\mathcal B}_q$ where $\mca_q$ is defined in Lemma \ref{lem:AqB} and ${\mathcal B}_q$ is defined in Lemma \ref{thm:lchix}. So combining with Lemma \ref{lem:AqB}, we have 
		\[\# \mce(q)\ll \exp\left(-\frac{c_2(\sigma)\log q}{\log_2 q}\right).\]
		We note that $\log L(\sigma,\chi)=R_y(\sigma,\chi)+O((\log q)^{-A})$, for all $\chi\in \mcb_q^c$. Then we get
		\begin{equation}\label{eq:chinotineq}
			\begin{split}
				&\fq\sum_{\chi({\rm mod}\;q)\atop \chi\notin{\mathcal E}(q)} |L(\sigma,\chi)|^z=\fq \sum_{\mca_q^c\cap\mcb_q^c}\exp\left(z\RE R_y(\sigma,\chi)+O((\log q)^{-A+\sigma})\right)\\
				=&\fq \sum_{\mca_q^c\cap\mcb_q^c}\exp\left(z\RE R_y(\sigma,\chi)\right)+O\left(\frac{1}{q(\log q)^{A-\sigma}}\sum_{\mca_q^c\cap\mcb_q^c}\exp\left(z\RE R_y(\sigma,\chi)\right)\right)\\
				=&\fq \sum_{\mca_q^c\cap\mcb_q^c}\exp\left(z\RE R_y(\sigma,\chi)\right)+O\left(\frac{1}{(\log q)^{A-\sigma}}\mbe(|L(\sigma,X)|^{\RE z})\right)
			\end{split}
		\end{equation}
		The last identity follows from Lemma \ref{lem:chiaqc}. On the other hand, since for all $\chi\in \mca_q^c$, we have 
		\[|R_y(\sigma,\chi)|\leq \frac{(\log q)^{1-\sigma}}{\log_2 q}.\]
		So 
		\begin{equation}\label{eq:acnotinbc}
			\fq\sum_{\chi\in \mca_q^c\atop \chi\notin \mcb_q^c}\exp\left(z\RE R_y(\sigma,\chi)\right)\ll \fq\sum_{\chi\in \mca_q^c\atop \chi\notin \mcb_q^c}\exp\left(\frac{\log q}{\log_2 q}\right)\ll \frac{\#\mcb_q}{q}\exp\left(\frac{\log q}{\log_2 q}\right).
		\end{equation}
		Since Lemma \ref{l8} implies that $\#\mcb_q\ll q^{\frac{1}{2}\left(\frac{9-6\sigma}{7-2\sigma}+1\right)}=q^{\frac{8-4\sigma}{7-2\sigma}}$, so Eq. \eqref{eq:acnotinbc} is bounded by 
		\[q^{\frac{1-2\sigma}{7-2\sigma}}\exp\left(\frac{\log q}{\log_2 q}\right)\ll q^{\frac{1-2\sigma}{14-4\sigma}}.\]
		Together with Eq. \eqref{eq:chinotineq}, we get 
		\[\fq\sum_{\chi({\rm mod}\;q)\atop \chi\notin{\mathcal E}(q)} |L(\sigma,\chi)|^z=\fq\sum_{\chi\in \mca_q^c}\exp\left(z\RE R_y(\sigma,\chi)\right)+O\left(\frac{1}{(\log q)^{A-\sigma}}\mbe(|L(\sigma,X)|^{\RE z}\right).\]
		Finally, the desired result just follows from Lemma \ref{lem:chiaqc} and Lemma \ref{lem:yuv}.
	\end{proof}
	
	Now we can prove the first part of Theorem \ref{thm1.2}.
	\begin{proof}[Proof of Theorem \ref{thm1.2} for $\frac12<\sigma<1$]
		Define
		$$
		I(\sigma,\tau):= \frac{1}{2\pi \mmi}\int_{\kappa-\mmi\infty}^{\kappa+\mmi\infty}{\mathbb E}\left(|L(\sigma,X)|^s\right)\mme^{-\tau s}\left(\frac{\mme^{\lambda s}-1}{\lambda s}\right)^N\frac{\mmd s}{s}
		$$
		and 
		$$ 
		J_q(\sigma,\tau):= \frac{1}{2\pi \mmi}\int_{\kappa-\mmi\infty}^{\kappa+\mmi\infty}\bigg(\frac{1}{q}\sum_{\chi({\rm mod}\;q)\atop \chi\notin{\mathcal E}(q)}|L(\sigma,\chi)|^s\bigg)\mme^{-\tau s}\left(\frac{\mme^{\lambda s}-1}{\lambda s}\right)^N\frac{\mmd s}{s}.
		$$
		Then we have
		$$\Psi(\tau)\le I(\sigma,\tau)\le\Psi(\tau-\lambda N)$$
		and 
		$$\Phi_q(\tau)+O(\delta(q))\le J_q(\sigma,\tau)\le\Phi_q(\tau-\lambda N)+O(\delta(q)),$$
		with
		$$
		\delta(q)=\exp\bigg(-c_0(\sigma)\frac{\log q}{\log_2 q}\bigg),
		$$
		for some positive constant $c_0(\sigma)$. Let $y=b_4(\log q)^\sigma$. Divide each integral of $I(\sigma,\tau)$ and $J_q(\sigma,\tau)$ into three parts:
		$$\int_{\kappa-\mmi\infty}^{\kappa+\mmi\infty}=\int_{\kappa-\mmi\infty}^{\kappa-\mmi y}+\int_{\kappa-\mmi y}^{\kappa+\mmi y}+\int_{\kappa-\mmi y}^{\kappa+\mmi\infty}.$$
		By Proposition \ref{prop4.1}, we have for the second part
		\begin{align*}
			\frac{1}{2\pi \mmi}\int_{\kappa-\mmi y}^{\kappa+ \mmi y}\bigg(\frac{1}{q}\sum_{\chi({\rm mod}\;q)\atop \chi\notin{\mathcal E}(q)}|L(\sigma,\chi)|^s-{\mathbb E}\left(|L(\sigma,X)|^s\right)\bigg)&\mme^{-\tau s}\left(\frac{\mme^{\lambda s}-1}{\lambda s}\right)^N\frac{\mmd s}{s}\\&\ll\frac{3^Ny}{(\log q)^{10}}{\mathbb E}\left(|L(\sigma,X)|^{\kappa}\right)\mme^{-\tau \kappa}.\end{align*}
		
		For the first and third parts, by Proposition \ref{prop4.1} we have
		\begin{equation*}
			\bigg(\int_{\kappa-\mmi\infty}^{\kappa-\mmi y}+ \int_{\kappa+\mmi y}^{\kappa+\mmi\infty}\bigg){\mathbb E}\left(|L(\sigma,X)|^s \right)\mme^{-\tau s}\left(\frac{\mme^{\lambda s}-1}{\lambda s}\right)^N\frac{\mmd s}{s}\ll \left(\frac{3}{\lambda y}\right)^N{\mathbb E}\left(|L(\sigma,X)|^{\kappa}\right)\mme^{-\tau \kappa},
		\end{equation*}
		and 
		\begin{equation*}
			\begin{aligned}
				\bigg(\int_{\kappa-\mmi\infty}^{\kappa-\mmi y}+ \int_{\kappa+\mmi y}^{\kappa+\mmi\infty} \bigg)       \left(\frac{1}{q}\sum_{\chi({\rm mod}\;q)\atop \chi\notin{\mathcal E}(q)}|L(\sigma,\chi)|^s\right)&{\rm e}^{-\tau s}\left(\frac{\mme^{\lambda s}-1}{\lambda s}\right)^N\frac{\mmd s}{s}\\
				&\ll \left(\frac{3}{\lambda y}\right)^N{\mathbb E}\left(|L(\sigma, X)|^{\kappa}\right)\mme^{-\tau \kappa}.
			\end{aligned}
		\end{equation*}
		Combining the above three inequalities, we have
		\begin{equation*}
			J_q(\sigma, \tau)- I(\sigma, \tau)\ll {\mathbb E}\left(|L(\sigma,X)|^{\kappa}\right)\mme^{-\tau \kappa}\left(\frac{3^N y}{(\log q)^{10}}+ \left(\frac{3}{\lambda y}\right)^N\right).
		\end{equation*}
		Since 
		$$\Psi(\tau)\asymp_\sigma\frac{1}{\tau^{\frac{1}{2(1-\sigma)}}(\log \tau)^{\frac{\sigma}{2(1-\sigma)}}}{\mathbb E}\left(|L(\sigma,X)|^{\kappa}\right)\mme^{-\tau \kappa},$$
		by choosing $N=\lfloor\log_2q\rfloor$ and $\lambda={\rm e}^{10}/y$, we have
		$$J_q(\sigma, \tau)- I(\sigma, \tau)\ll\frac{1}{(\log q)^5}\Psi(\tau).$$
		Since Theorem \ref{thm1.1} gives
		\[\Psi(\tau\pm \lambda N)=\Psi(\tau)\bigg(1+O\bigg(\frac{(\tau\log \tau)^{\frac{\sigma}{1-\sigma}}\log_2 q}{(\log q)^{\sigma}}\bigg)\bigg),\]
		we have 
		
		\begin{align*}&\Phi(\tau)\le J_q(\sigma,\tau)+O(\delta(q))\le I(\sigma, \tau)+\frac{1}{(\log q)^5}\Psi(\tau)+O(\delta(q))\\&\le\Psi(\tau-\lambda N)+\frac{1}{(\log q)^5}\Psi(\tau)+O(\delta(q))\\
			&\le\Psi(\tau)\bigg(1+O\bigg(\frac{(\tau\log \tau)^{\frac{\sigma}{1-\sigma}}\log_2 q}{(\log q)^{\sigma}}\bigg)\bigg)+O(\delta(q)),\end{align*}
		and
		\begin{align*}&\Phi(\tau)\ge J_q(\sigma,\tau+\lambda N)+O(\delta(q))\ge I(\sigma, \tau+\lambda N)+\frac{1}{(\log q)^5}\Psi(\tau)+O(\delta(q))\\&\ge\Psi(\tau+\lambda N)+\frac{1}{(\log q)^5}\Psi(\tau)+O(\delta(q))\\
			&\ge\Psi(\tau)\bigg(1+O\bigg(\frac{(\tau\log \tau)^{\frac{\sigma}{1-\sigma}}\log_2 q}{(\log q)^{\sigma}}\bigg)\bigg)+O(\delta(q)).\end{align*}
		Thus we have 
		\[\Phi(\tau)=\Psi(\tau)\bigg(1+O\bigg(\frac{(\tau\log \tau)^{\frac{\sigma}{1-\sigma}}\log_2 q}{(\log q)^{\sigma}}\bigg)\bigg)\]
		by the trivial bound $\Psi(\tau)\gg\sqrt{\delta(q)}$.
	\end{proof}

	\subsection{\blue{When $\sigma=1$}}

	\begin{lem}\label{lem:aqsigma1}
		Let $y=(\log q)^a$ for some $a\geq 1$. Denote $\mca_{1,q}=\{\chi({\rm mod}\  q)\, : \,|R_y(1,\chi)|>\log_3 q\}$. Then we have 
		\[\frac{\# \mca_{1,q}}{q}\ll \exp\left(-\frac{c\log q\log_3 q}{\log_2 q}\right),\]
		for some constant $c>0$.
	\end{lem}

	\begin{proof}
		Similar to the proof of Lemma \ref{lem:1qchi2k}, we have 
		\[\fq\sum_{\chi({\rm mod}\ q)}|R_y(1,\chi)|^{2k}\ll \frac{(3\log_2 k)^{2k}}{q}+\left(\frac{3}{\log k}\right)^{2k}.\]
		So we get
		\[\# \mca_{1,q}(\log_3 q)^{2k}\leq |R_y(1,\chi)|^{2k}\ll (3\log_2 k)^{2k}+q\left(\frac{3}{\log k}\right)^{2k}.\]
		Now by taking $k=[\log q/\log_2 q]$, we get the desired result.
	\end{proof}
	
	Then we can prove the second part of Theorem \ref{thm1.2} for $\sigma=1$.

	\begin{proof}[Proof of Theorem \ref{thm1.2} for $\sigma=1$]
		The proof is similar to the case of $1/2<\sigma<1$. Fix $A>1$, let $\mce_1(q)=\mca_{1,q}\cup\mcb_q$, then Lemma \ref{lem:aqsigma1} shows that $\#{\mathcal E}_1(q)\leq q\exp\left(-\frac{b_3'\log q\log_3 q}{\log_2q}\right)$ for some positive constant $b_3'$. With the same argument as in Proposition \ref{prop4.1}, one can show that there exists a positive constant $b_4'$ such that for all complex numbers $z$ with $|z|\le  b_4'\log q$ we have 
		\begin{equation}\label{eq:lzb4y}
			\frac{1}{q}\sum_{\chi({\rm mod}\;q)\atop \chi\notin{\mathcal E}_1(q)}|L(1,\chi)|^z={\mathbb E}(|L(1,X)|^z)+O\bigg(\frac{{\mathbb E}(|L(1,X)|^{{\rm Re\,}z})}{(\log q)^{A-1}}\bigg).
		\end{equation}
		Then we take $y=\frac{b_4'}{2}\log q$. Since $\kappa$ satisfies $M'(\kappa)=\gamma+\log \tau$, so we get $\kappa\leq y$. Let $s=\kappa+\mmi t$ with $|t|\leq y$. Then Eq. \eqref{eq:lzb4y} holds for $z=s$. Now we define
		$$
		I_1(\tau):= \frac{1}{2\pi \mmi}\int_{\kappa-\mmi\infty}^{\kappa+\mmi\infty}{\mathbb E}\left(|L(\sigma,X)|^s\right)(\mme^\gamma\tau)^{- s}\left(\frac{\mme^{\lambda s}-1}{\lambda s}\right)^N\frac{\mmd s}{s},
		$$
		and 
		$$ 
		J_{1,q}(\tau):= \frac{1}{2\pi \mmi}\int_{\kappa-\mmi\infty}^{\kappa+\mmi\infty}\bigg(\frac{1}{q}\sum_{\chi({\rm mod}\;q)\atop \chi\notin{\mathcal E}_1(q)}|L(\sigma,\chi)|^s\bigg)(\mme^\gamma\tau)^{- s}\left(\frac{\mme^{\lambda s}-1}{\lambda s}\right)^N\frac{\mmd s}{s}.
		$$
		By Lemma  \ref{SmoothPerron}, we have 
		\[\Psi_1(\tau)\le I_1(\tau)\le\Psi_1(\tau-\lambda N)\]
		and
		\[\Phi_{1,q}(\tau)\leq J_{1,q}(\sigma)+O\left(\exp\left(-\frac{c\log q\log_3 q}{\log_2 q}\right)\right)\le\Phi_{1,q}(\tau-\lambda N).\]
		Along the same line of the proof for the case of $\frac12<\sigma<1$, we can deduce that 
		\[J_{1,q}(\tau)-I_1(\tau)\ll\bigg(\frac{3^Ny}{(\log q)^{A-1}}+\bigg(\frac{3}{\lambda y}\bigg)^N\bigg){\mathbb E}(|L(1,X)|^\kappa)({\rm e}^\gamma\tau)^{-\kappa}.\]
		Since $M'(\kappa)=\log\tau+\gamma$, following Proposition \ref{asyPsi1}, we have 
		\[\Psi_1(\tau)\asymp\sqrt{\frac{\log \kappa}{\kappa}}{\mathbb E}(|L(1,X)|^\kappa)({\rm e}^\gamma\tau)^{-\kappa}\asymp\sqrt{\frac{\tau}{{\rm e}^\tau}}{\mathbb E}(|L(1,X)|^\kappa)({\rm e}^\gamma\tau)^{-\kappa}.\]
		
		We also note that Corollary \ref{coro4.6} gives 
		\[\Psi_1(\tau\pm\lambda N)=\Psi_1(\tau)\left(1+O\left(\frac{\mme^{\tau}\log_2 q}{\tau\log q}\right)\right).\]
		So with the same argument, we show that 
		\[\Phi_{1,q}(\tau)=\Psi_1(\tau)\left(1+O\left(\frac{\mme^{\tau}\log_2 q}{\tau\log q}\right)\right).\]

		
	\end{proof}

	\section{\blue{Discrepancy bounds:  Proof of Theorem \ref{thm1.3}}}\label{chap5}
	In order to prove Theorem \ref{thm1.3}, we need an approximation of character function in \cite{LLR19} which originally comes from Selberg. To present the result, we first give some notations. For $u\in[0,1]$, let 
	\[G(u)=\frac{2u}{\pi}+2u(1-u)\cot(\pi u).\]
	For any $a,b\in \mbr$, let 
	\[f_{a,b}(u)=\frac{1}{2}(\mme^{-2\pi \mmi au}-\mme^{-2\pi \mmi bu}).\]
	\begin{lem}\label{lem:llr6.2}
		Let $\mcr=\{z=x+\mmi y\in\mbc\, :\, a_1<x<a_2,b_1<y<b_2\}$ and $T>0$ be a real number. Then for any $z=x+\mmi y$ we have 
		\[\mathbf{1}_\mcr(z)=W_{T,\mcr}(z)+O\left(I(T(x-a_1))+I(T(x-a_2))+I(T(y-b_1))+I(T(y-b_2))\right),\]
		where $W_{T,\mcr}(z)$ is equal to
		\[\frac{1}{2}\RE\int_0^T\int_0^TG\left(\frac{u}{T}\right)G\left(\frac{v}{T}\right)f_{a_1,a_2}(u)\left(\mme^{2\pi \mmi(ux-vy)}\ol{f_{b_1,b_2}(v)}-\mme^{2\pi \mmi(ux+vy)}f_{b_1,b_2}(v)\right)\frac{\mmd u}{u}\frac{\mmd v}{v}.\]
		and $I(x)=\frac{\sin^2\pi x}{(\pi x)^2}$.
	\end{lem}
	\begin{proof}
		See~\cite[Lemma $6.2$]{LLR19}.
	\end{proof}
	
	\begin{lem}\label{lem:mbeit}
		Let $1/2<\sigma\leq 1$ be fixed and $T>0$, then
		\[\mbe(I(T(\RE\log L(\sigma, X)-s)))\ll\frac{1}{T},\quad \text{and}\quad \mbe(I(T(\IM\log L(\sigma, X)-s)))\ll\frac{1}{T}\]
		hold uniformly in $s\in \mbr$.
	\end{lem}
	\begin{proof}
		See~\cite[Theorem $1.1$]{LLR19}.
	\end{proof}
	
	\begin{proof}[Proof of Theorem \ref{thm1.3}]
		Let $\mcr$ be any rectangle with sides parallel to the coordinate axis and $\mcr'=\mcr\cap[-\log_2 q,\log_2 q]^2$. Then by taking $\tau=\log_2 q$ in Theorem \ref{thm1.1}, we have
		\[\Phi_q(\mcr)=\Phi_q(\mcr')+O\left(\frac{1}{(\log q)^A}\right) \quad \text{and}\quad \Psi(\mcr)=\Psi(\mcr')+O\left(\frac{1}{(\log q)^A}\right)\quad\forall A>0.\]
		So we can reduce to the case $\mcr\subseteq [-\log_2 q,\log_2 q]^2$. When $1/2\sigma<1$, we choose $T=c(\log q)^{\sigma}$ where the constant $c$ is the same as in Lemma \ref{thm:lchix}. By Lemma \ref{lem:llr6.2}, we see that 
		\begin{equation}\label{eq:phiqmcr}
			\begin{split}
				\Phi_q(\mcr)=&\fq\sum_{\chi({\rm mod}\ q)}W_{T,\mcr}(\log L(\sigma,\chi))+O\bigg( \fq\sum_{\chi({\rm mod}\ q)}I(T(\RE \log L(\sigma,\chi)-a_1))\\&+I(T(\RE \log L(\sigma,\chi)-a_2))
				+I(T(\IM \log L(\sigma,\chi)-b_1))+I(T(\IM \log L(\sigma,\chi)-b_2))\bigg).
			\end{split}
		\end{equation}
		Fix any positive real number $A>3\sigma$ in Lemma \ref{thm:lchix}, we have
		\begin{align*}
			&\fq\sum_{\chi({\rm mod}\ q)}W_{T,\mcr}(\log L(\sigma,\chi))\\
			=&\mbe(W_{T,\mcr}(\log L(\sigma,X)))+O\left(\frac{1}{(\log q)^A}\int_0^T\int_0^TG\left(\frac{u}{T}\right)G\left(\frac{v}{T}\right)|f_{a_1,a_2}(u)f_{b_1,b_2}(v)|\frac{\mmd u}{u}\frac{\mmd v}{v}\right)
		\end{align*}
		We note that $0\leq G(u)\leq 2/\pi$ when $0\leq u\leq 1$ and $|f_{a,b}(u)|\leq \pi u|b-a|$. So the error term is bounded by 
		\[\frac{1}{(\log q)^A}\int_0^T\int_0^TG\left(\frac{u}{T}\right)G\left(\frac{v}{T}\right)|f_{a_1,a_2}(u)f_{b_1,b_2}(v)|\frac{\mmd u}{u}\frac{\mmd v}{v}\ll \frac{(a_1-a_2)(b_1-b_2)T^2}{(\log q)^A}.\]
		By our assumption, $|(a_1-a_2)(b_1-b_2)|\leq 4(\log_2 q)^2$, so the error term is bounded by $\frac{1}{\log q)^\sigma}$. So we get
		\begin{equation}\label{eq:wtr}
			\fq\sum_{\chi({\rm mod}\ q)}W_{T,\mcr}(\log L(\sigma,\chi))=\mbe(W_{T,\mcr}(\log L(\sigma,X)))+O\left(\frac{1}{(\log q)^\sigma}\right).
		\end{equation}
		
		On the other hand, since $\mbe(\mathbf{1}_{\mcr}(\log L(\sigma, X)))$ is exactly the probability of $\log L(\sigma,X)\in\mcr$, so by Lemma \ref{lem:llr6.2} and Lemma \ref{lem:mbeit}, we have 
		\begin{equation}\label{eq:wtrplx}
			\mbe(W_{T,\mcr}(\log L(\sigma,X)))=\Psi_q(\mcr)+O\left(\frac{1}{T}\right).
		\end{equation}
		So combining Eq. \eqref{eq:wtr} and Eq. \eqref{eq:wtrplx}, we get 
		\begin{equation}\label{eq:wtrpsi}
			\fq\sum_{\chi({\rm mod}\ q)}W_{T,\mcr}(\log L(\sigma,\chi))=\Psi_q(\mcr)+O\left(\frac{1}{T}\right).
		\end{equation}
		Finally, using again Lemma \ref{thm:lchix} and Lemma \ref{lem:mbeit}, we get
		\begin{equation}\label{eq:itre}
			\fq\sum_{\chi({\rm mod}\ q)}I(T(\RE \log L(\sigma,\chi)-s))\ll \frac{1}{T},
		\end{equation}
		and 
		\begin{equation}\label{eq:itim}
			\fq\sum_{\chi({\rm mod}\ q)}I(T(\IM \log L(\sigma,\chi)-s))\ll \frac{1}{T},
		\end{equation}
		By combining Eq. \eqref{eq:phiqmcr}, \eqref{eq:wtrplx}, \eqref{eq:itre} and \eqref{eq:itim}, we get 
		\[\Phi_q(\mcr)=\Psi_q(\mcr)+O\left(\frac{1}{T}\right).\]
		When $\sigma=1$, we choose $T=\frac{\log q}{50(\log_2q)^2}$. Following Lemma \ref{lem:laml1}, with the same argument as above, we get the desired result.
	\end{proof}

	\section*{Acknowledgements}
	The research of the third author was supported by Fundamental Research Funds for the Central Universities (Grant No. 531118010622).



\end{document}